\documentclass{article}
\usepackage{amsmath, amsthm, amsfonts}
\usepackage{amssymb} 
\usepackage{amsfonts}
\usepackage{graphicx}
\usepackage{tikz}
\usepackage{stmaryrd} 
\usepackage{float}
\usepackage{subfig}
\usepackage{booktabs}

\makeatletter
\newsavebox{\@brx}
\newcommand{\llangle}[1][]{\savebox{\@brx}{\(\m@th{#1\langle}\)}%
  \mathopen{\copy\@brx\kern-0.5\wd\@brx\usebox{\@brx}}}
\newcommand{\rrangle}[1][]{\savebox{\@brx}{\(\m@th{#1\rangle}\)}%
  \mathclose{\copy\@brx\kern-0.5\wd\@brx\usebox{\@brx}}}
\makeatother

\newtheorem{thm}{Theorem}

\newtheorem{rem}{Remark}

\newcommand{\q}{Q}
\newcommand{\bx}{\ensuremath{x}}
\newcommand{\okappa}{\ensuremath{\bar{\kappa}}}
\newcommand{\bn}{\mathbf{n}}

\def\jump#1{\left \llbracket #1 \right \rrbracket}

\usepackage{authblk}
\title{\emph{A Posteriori} Analysis and Efficient Refinement Strategies for the Poisson-Boltzmann Equation}
\author{Jehanzeb H.~Chaudhry  \thanks{Department of Mathematics and Statistics, The University of New Mexico, Albuquerque, NM 87131. Email: jehanzeb@unm.edu}}

\usepackage{amsopn}

\begin{document}


\maketitle

\begin{abstract}
The Poisson-Boltzmann  equation (PBE) models the electrostatic interactions of charged bodies such as molecules and proteins in an electrolyte solvent. The PBE is a challenging equation to solve numerically due to the presence of singularities, discontinuous coefficients and boundary conditions. Hence, there is often large error in the numerical solution of the PBE that needs to be quantified. In this work, we use adjoint based \emph{a posteriori} analysis to accurately quantify the error in an important quantity of interest, the solvation free energy, for the finite element solution of the PBE. We identify various sources of error and propose novel refinement strategies based on \emph{a posteriori} error estimates.
\end{abstract}



\section{Introduction}
\label{sec:intro}

Electrostatic interactions play a critical role in determining
macroscopic properties of dielectric biomolecular systems, such as solvation free energy
and binding affinities~\cite{DaMc90,Ko2006,ViMa2005}. The
Poisson-Boltzmann equation (PBE) has been widely used for modeling the electrostatic interactions of charged bodies such as molecules and proteins in electrolyte solvents.  The PBE was introduced decades ago~\cite{Go1910,Ch1913}, and we refer to the classical texts \cite{McQuarrie76,Tanford61} for its derivation.

The focus of this article is robust error estimation and refinement strategies for computing a quantity of interest (QoI), such as the solvation free energy, from the solution of the PBE. The PBE is a challenging equation to solve numerically and numerous computational methods and software packages have been derived for its solution~\cite{
GSH87,
Ra90, 
YoLe90,
VGS92, 
Zho93, 
BWSH95, 
DaMc89, 
HoSa95, 
LDM92,
Bashford97,
CHX07,
Ortt77, 
CoFr97, 
HBW00, 
BHW00, 
SMN02,
WYQ05,
LZMJ05,
BCCO09,
HMY+2012}. In this article we follow the approach in \cite{GYW2007,HMY+2012} to solve the PBE using a three term splitting method which accounts for the well-posedness of the continuum problem as well as avoiding amplification of numerical rounding errors. However, even this method, like all numerical methods, often has significant errors in the computation of the QoI and this error needs to be accurately estimated from computed information for reliable use of the PBE in biophysics, biochemistry, medical and other science and engineering fields~\cite{Estep:Larson:00,Fish:2009}.

In this article we employ adjoint based \emph{a posteriori} analysis to accurately
quantify the error in a QoI computed from the numerical solution of the PBE.  Adjoint based error estimation is
widely used for a host of numerical methods including finite elements, finite difference, 
time integration, multi-scale simulations and inverse problems~\cite{estep_sinum_95,eehj_book_96,Estep:Larson:00,
  AO2000,BangerthRannacher,barth04,beckerrannacher,
  Giles:Suli:02, YCLP2004, CEG+2015, jdt13,  VRAS2015}. The error estimate weights
computable residuals of the numerical solution with the
solution of an adjoint problem to quantify the accumulation and propagation of
error.  The resulting estimates have the useful feature that the total error is decomposed as a sum of contributions from various aspects of the discretization and therefore provide insight in to the effect of different choices for the parameters controlling the discretization. Thus, we not only quantify the error using adjoint based a posteriori analysis, we also partition the error to identify contributions from various sources of error. For example, we can identify if the boundary discretization or the interior discretization is the major source of error.


Since the error in the numerical solution of the PBE is often significant, there are a number of adaptive refinement strategies proposed for obtaining accurate solutions of the PBE~\cite{BCCO09,HBW00,HMY+2012,VoSc97,ABC+2011}. Most of the adaptive algorithms are based on controlling the error in global norms and some of the algorithms are shown to be provably convergent~\cite{CHX07}. However, if the goal of the numerical computation is accurate approximation of the QoI, then a refinement strategy based on solution residuals weighted by the adjoint information is an appealing option. In this paper we propose refinement strategies based on the relative contribution to the error of a discretization choice.  The adjoint based analysis and its partitioning of the error suggests novel refinement strategies for obtaining accurate estimates of the QoI from the numerical solution of the PBE.

Adaptive refinement using adjoint based analysis and optimal multilevel preconditioning  for the PBE are developed previously in \cite{ABC+2011}.
However, the analysis and results of this article differ significantly from that paper. The focus of that paper was adaptive refinement for the linearized PBE using the two term splitting~\cite{CHX07}, whereas we focus on the three term splitting for both the linearized and the nonlinear PBE~\cite{GYW2007,HMY+2012}. Moreover, our aim is to derive accurate error estimates in the QoI. While \cite{ABC+2011} derived an error estimate for the QoI for the three term splitting, no numerical experiments were performed for the three term splitting. Even for the two term splitting no numerical results indicating the accuracy of the derived estimates were shown and instead the focus was on adaptive refinement. In addition, the adjoint problem for the three term splitting derived in \cite{ABC+2011} leads to an ill-posed problem as we discuss in \S \ref{sec:apost_analysis}.  In this work, we  not only derive  an estimate for the three term splitting based on the correct formulation of the adjoint operator, we also decompose the error so that the various sources of error and their relative contributions are also available. 
Moreover, the error estimate derived in \cite{ABC+2011} assumes that the continuum and discrete solutions satisfy the boundary conditions exactly. While this assumption may be justifiable for the results in the two term splitting in \cite{ABC+2011}, we point out the importance of the role of boundary condition for the harmonic component of the three term splitting. This boundary condition is defined on the interface between the solvent and molecular regions, and hence impacts the computation of the QoI significantly. Finally, we propose a fundamentally different refinement strategy since the standard goal oriented refinement strategy employed in \cite{ABC+2011} appears to be sub-optimal. Adaptive refinement for obtaining accurate values of a QoI is a challenging task as the error contributions of an individual element may be positive or negative leading to significant cancellation of error. In \cite{ABC+2011}, the refinement strategy takes the absolute value of error contributions and applies the principle of equidistribution for marking elements for refinement. This strategy ignores the cancellation of error, and hence the resulting adaptive algorithm may have less than desirable convergence properties. This drawback is overcome in \cite{ABC+2011} by defining a somewhat ad-hoc error indicator.  On the other hand, this article decomposes the error into different contributions and use this information to devise adaptive schemes to target the discretization choices which have the most significant effect on error.

The rest of the paper is organized as follows. Section \ref{sec:pbe} introduces the PBE, its linearized and nonlinear versions, weak forms and a finite element method to solve it. Section \ref{sec:apost_analysis}  performs adjoint based \emph{a posteriori} analyses for both the linearized and nonlinear PBE. In particular, a well-posed adjoint problem for the three-term PBE  and error representations are derived. Section \ref{sec:refine_strategies_apost_ideas} discusses refinement strategies based on \emph{a posteriori} error estimates. Numerical experiments are presented in Section \ref{sec:num_exps}, which illustrate the accuracy of the estimates as well as  the efficacy of employing refinement strategies which target specific sources of error.  Section \ref{sec:conclusions} presents conclusions.

\section{The Poisson-Boltzmann Equation}
\label{sec:pbe}

\subsection{The nonlinear Equation and its dimensionless form}
The Poisson-Boltzmann equation models the electrostatic activity between
molecules in an ionic solvent.  In this model, it is assumed that the ions in the
solvent are distributed according to the Boltzmann distribution and that the
potential of the mean force on a particle is simply the charge of the ion times
the electrostatic potential.  For a 1:1 electrolyte solvent (e.g. NaCl), the nonlinear  Poisson-Boltzmann 
equation is~\cite{FBM02,BBC06},
\begin{equation}
	\begin{cases}
\begin{aligned}
-\nabla \cdot (\epsilon(\bx) \nabla \tilde{u}(\bx)) + \okappa^2(\bx) \left(
\frac{k_BT}{e_c} \right) \sinh\left( \frac{e_c \tilde{u}(\bx)}{k_B T}\right)  = 4
\pi \sum_{i=1}^{m} \q_i \delta(\bx - \bx_i), \label{eq:PBE_original} \\
\lim_{\|\bx\|\rightarrow \infty} \tilde{u}(\bx)  = 0.
\end{aligned}
\end{cases}
\end{equation}
%
Here, $\tilde{u}$ is the unknown electrostatic potential, $\epsilon$ is the 
dielectric coefficient, $\okappa(\bx)$ is the modified Debye-H\"uckel
parameter which  describes the accessibility of the solvent
to the solute,  $k_B$ is the Boltzmann constant, $e_c$ is the charge on a proton  and $T$ is the 
temperature.  
Moreover, the solute contains a total of $m$ fixed point 
charges, with the $i$th charge $\q_i$ centered at position $x_i$. The
resulting distribution is a linear combination of Dirac delta 
functions $\delta( \bx - x_i )$.

The domain for the problem $\mathbb{R}^3$, is subdivided into a molecular 
region, $\Omega_m$, a solvent region $\Omega_s^\infty$, and an interface between 
the two denoted by $\Gamma$. The solute is surrounded by solvent, which 
is represented as a continuum over the subdomain 
$\Omega_s^\infty = \mathbb{R}^3 \backslash \overline{\Omega}_m$.
The subdomains for a typical biomolecular solute are shown in 
Figure~\ref{fig:pbe_domains_pic} which has been adopted from \cite{BCCO09}.
\begin{figure}[!ht]
\centering
\includegraphics[width=.6\textwidth]{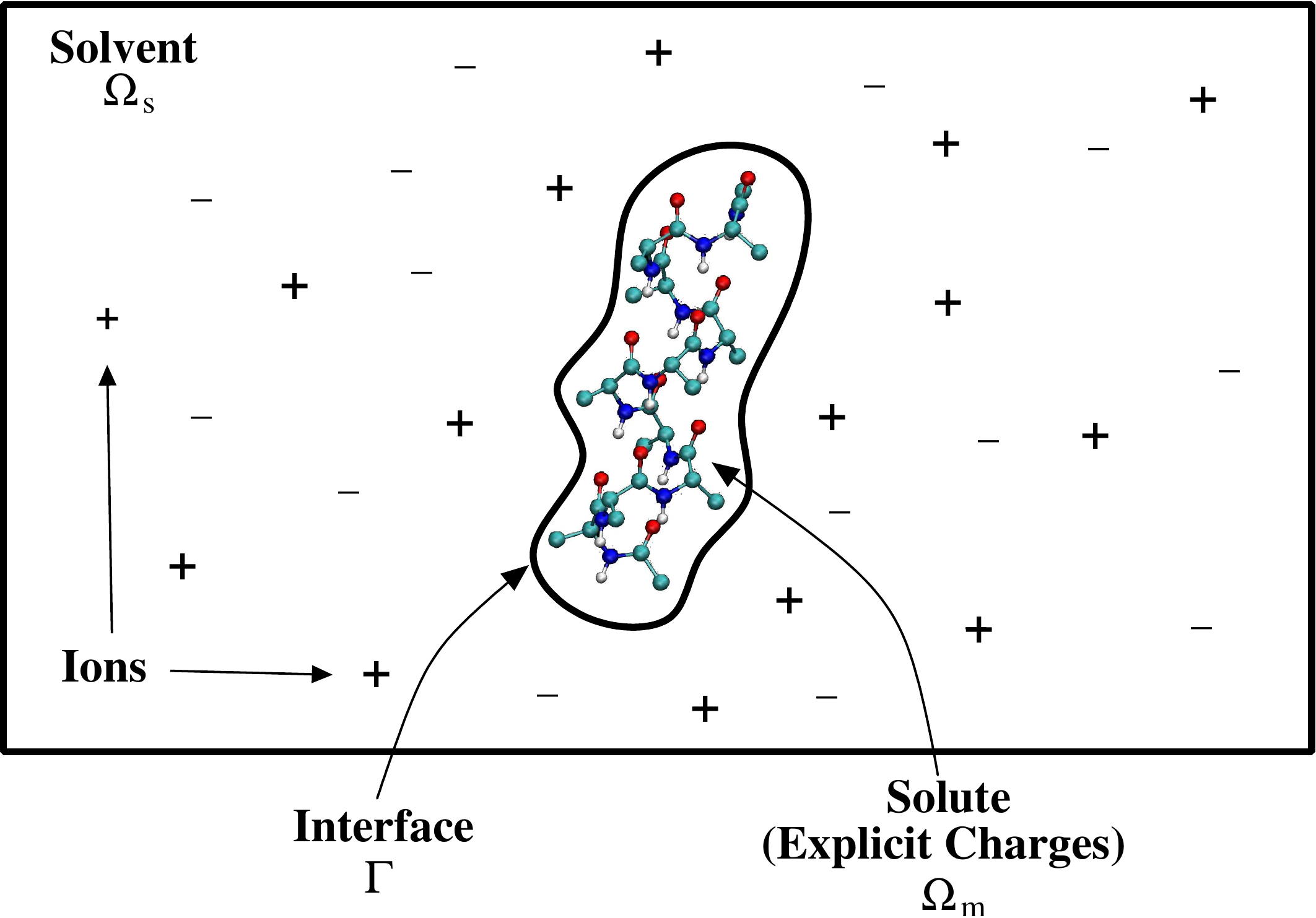}
\caption{Subdomains for the Poisson-Boltzmann equation}
\label{fig:pbe_domains_pic}
\end{figure}
The dielectric coefficient $\epsilon(\bx)$ and modified Debye-H\"uckel
parameter $\okappa(\bx)$ are defined on $\Omega_m \cup \Omega_s^\infty$ by the
piecewise constant functions
\begin{equation}
\epsilon(\bx) =
\begin{cases}
\epsilon_m & \quad \bx \in \Omega_m\\
\epsilon_s & \quad \bx \in \Omega_s^\infty
\end{cases}
\quad
\textnormal{and}
\quad
\okappa^2(\bx) =
\begin{cases}
0 & \bx \in \Omega_m\\
\okappa^2_s = \epsilon_s \frac{8 \pi N_A e_c^2}{1000 k_B T}I_s &
\bx \in \Omega_s^\infty
\end{cases}.
\end{equation}
Here, $\epsilon_m$ and $\epsilon_s$ are positive constants and 
$N_A$ is Avogadro's number. The ionic
strength $I_S$ is a physical parameter which varies depending on
the solvent. 

Numerical simulations are not feasible on the unbounded solvent domain, $\Omega_s^\infty$, and hence it
is  truncated at a finite radius from the ``center'' of the molecule, to form a bounded domain $\Omega_s$. Dirichlet boundary 
conditions are 
imposed to capture the asymptotic behavior of the solution on an unbounded 
domain. Combining this with the change of variables, 
$u(\bx) = e_c \tilde{u}(\bx)/k_B T$, results in a dimensionless 
Poisson-Boltzmann equation on the spherical domain 
$\Omega = \Omega_m \cup \Omega_s \cup \Gamma$:
\begin{equation}
\label{eq:PBE_dimless} 
\begin{cases}
\begin{aligned}
-\nabla \cdot (\epsilon(\bx) \nabla u(\bx))
+ \okappa^2(\bx) \sinh({u(\bx)} )
= \frac{4 \pi e_c}{k_B T} \sum_{i=1}^{m} \q_i \delta(\bx - \bx_i), 
\quad x& \in \Omega,  \\
u(\bx)  = g(\bx), 	\, \, \, x& \in \partial \Omega.
\end{aligned}
\end{cases}
\end{equation}

The boundary conditions are prescribed using a linear combination of Helmholtz 
Green's functions~\cite{BCCO09},
\begin{equation}
g = \frac{e_c}{k_B T} \sum_{i=1}^{m} \frac{\q_i}{\epsilon_s \vert \bx - \bx_i \vert} \exp\left(\frac{-\okappa_s \vert \bx - \bx_i \vert}{\sqrt{\epsilon_s}}\right).
\end{equation}

\subsection{Weak form based on three term splitting}
We denote by $L_2(\Omega)$ as the space of square integrable functions,
 $H^1(\Omega)$ as the space of functions having an integrable (weak) derivative,  $H^1_0(\Omega)$ as the subspace of $H^1(\Omega)$ of functions satisfying homogeneous Dirichlet boundary conditions (in the sense of the trace operator) and $H^{-1}$ as the dual space of $H^1_0(\Omega)$.
The right hand side of \eqref{eq:PBE_dimless} contains $\delta$ functions, which are unbounded linear functionals over the space $H^{1}_0$ and hence a well-posed weak form cannot be derived directly from \eqref{eq:PBE_dimless}. To overcome this problem, two and three term splittings of the PBE have been proposed~\cite{ZPVK96,CHX07,HMY+2012}. The two and three term splitting are equivalent mathematically, however, the three term splitting is numerically more desirable~\cite{HMY+2012}. The three term splitting decomposes the function $u$ as
\begin{equation}
\label{eq:pbe_3_term_split}
u = \begin{cases}
u^s + u^h + u^r \quad &\text{ in } \quad \Omega_m,\\
u^r \quad &\text{ in } \quad \Omega_s,
\end{cases}
\end{equation}
where $u^s$, $u^h$ and $u^r$ are the singular, harmonic and regular components respectively. The singular function $u^s$ is the solution of the following Poisson equation
\begin{equation}
\begin{cases}
\begin{aligned}
- \nabla  \cdot \epsilon_m \nabla u^s = \frac{4 \pi e_c}{k_B T} \sum_{i=1}^{m} \q_i \delta(\bx - \bx_i),\\
u^s(\infty) = 0.
\end{aligned}
\end{cases}
\end{equation}
Recognizing that the singular component is the Green's function of the Laplace operator leads to an analytical expression for $u^s$ as
\begin{equation}
u^s(\bx) = \frac{e_c}{\epsilon_m k_B T} \sum_{i=1}^{m} \frac{\q_i}{|\bx - \bx_i|}.
\end{equation}
The harmonic component $u^h$ is the solution to
\begin{equation}
	\label{eq:harmonic_comp}
	\begin{cases}
	\begin{aligned}
	\nabla^2 u^h = 0 \quad &\text{ in } \Omega_m,\\
	u^h = -u^s  &\text{ on } \Gamma.
	\end{aligned}
	\end{cases}
\end{equation}
The regular component $u^r$ satisfies
\begin{equation}
\label{eq:pbe_regular_strong_nonlin}
\begin{cases}
\begin{aligned}
-\nabla \cdot  (\epsilon (\bx) \nabla u^r(x)) + \okappa^2(\bx) \sinh(u^r(\bx)) = 0 \quad \text{ in } \Omega_m \cup \Omega_s,\\
\jump{u^r}_\Gamma = 0,\\
\jump{\epsilon(\bx)\frac{\partial {u^r}(x)}{\partial \bn}}_\Gamma = -\epsilon_m \frac{\partial u^s + u^h}{\partial \bn} \\
u = g \quad \text{ on } \partial \Omega,
\end{aligned}
\end{cases}
\end{equation}
where the jump at the interface is defined as
\begin{equation}
\jump{v(x)}_\Gamma =
\lim_{\alpha
\rightarrow 0^+}  v(x + \alpha \bn)- v(x - \alpha \bn),
\nonumber
\end{equation}
with $\bn$ as the unit normal to the interface $\Gamma$, pointing outward from $\Omega_m$. The condition involving the jump in the normal derivative of $u^r$ arises by substituting \eqref{eq:pbe_3_term_split} in \eqref{eq:PBE_dimless}, using the definitions of $u^s$  and $u^h$, and the fact that for the solution $u$ of \eqref{eq:PBE_dimless} we have $\jump{\epsilon(\bx)\frac{\partial {u}(x)}{\partial \bn}} = 0$. Sometimes the nonlinear PBE is linearized by the assumption $\sinh(u) \approx u$ leading to the dimensionless linearized PBE. We can write both the linear and nonlinear versions as
\begin{equation}
\label{eq:pbe_regular_strong}
\begin{cases}
\begin{aligned}
-\nabla \cdot  (\epsilon (\bx) \nabla u^r(x)) + \okappa^2(\bx) N(u^r(\bx)) = 0 \quad \text{ in } \Omega_m \cup \Omega_s,\\
\jump{u^r}_\Gamma = 0,\\
\jump{\epsilon(\bx)\frac{\partial {u^r}(x)}{\partial \bn}}_\Gamma = -\epsilon_m \frac{\partial u^s + u^h}{\partial \bn}, \\
u = g \quad \text{ on } \partial \Omega.
\end{aligned}
\end{cases}
\end{equation}
where
\begin{equation}
	N(u^r(\bx)) = 
	\begin{cases}
	\begin{aligned}
	&\sinh (u^r(\bx)) \qquad &&\text{for nonlinear PBE},\\
	&u^r (\bx) \quad &&\text{for linearized PBE}.
	\end{aligned}
	\end{cases}
\end{equation}

\subsection{Weak forms}
We define the affine spaces 
\begin{equation}
H^1_{u^s}(\Omega_m) := \lbrace v \in H^1(\Omega_m) : v(x) = -u^s \text{ on } \Gamma\rbrace
\end{equation}
and
\begin{equation}
	H^1_g (\Omega) := \lbrace v \in H^1(\Omega): v = g \text{ on } \Gamma, \alpha \leq v \leq \beta \text{ in } \Omega \rbrace.
\end{equation}
Here $\alpha$ and $\beta$ are positive constants used to control the nonlinear $\sinh$ term, see \cite{HMY+2012} for details. The weak form for the three term split PBE, \eqref{eq:harmonic_comp} and \eqref{eq:pbe_regular_strong}, is to find 
$(u^h,u^r) \in H^1_{u^s}(\Omega_m) \times H^1_g(\Omega)$ such that
\begin{equation}
	\label{eq:pbe_3_term_weak_form_1}
	\begin{cases}
	\begin{aligned}
	&(\epsilon \nabla u^h, \nabla w)_m = 0\\
	&(\epsilon \nabla u^r, \nabla v) + (\okappa^2 N(u^r), v) + \langle \epsilon_m 
	\frac{\partial u^h}{\partial \bn} , v \rangle_\Gamma = - \langle \epsilon_m 
	\frac{\partial u^s}{\partial \bn} , v \rangle_\Gamma
	\end{aligned}
	\end{cases}
\end{equation}
for all $(w,v) \in H^1_0(\Omega_m) \times H^1_0(\Omega)$. Here we used the notation $(a, b) = \int_\Omega a b \, \mathrm{d}x$, $(a, b)_m = \int_{\Omega_m} a b \, \mathrm{d}x$ and $\langle a,b \rangle_\Gamma = \int_\Gamma a b \, \mathrm{d}x$ to represent the standard $L_2$ inner products over $\Omega$, $\Omega_m$ and $\Gamma$ respectively. The existence and uniqueness of the weak solution is shown in \cite{HMY+2012}. The weak form \eqref{eq:pbe_3_term_weak_form_1} is a one-way coupled system; we  first solve for $u^h$ and then use it to compute $u^r$. Now using the Green's identity 
\begin{equation}
\langle \epsilon_m  \frac{\partial u^h}{\partial \bn} , v \rangle = (\epsilon \nabla^2 u^h, v)_m + (\epsilon \nabla u^h, \nabla v)_m
\end{equation}
and \eqref{eq:harmonic_comp} in \eqref{eq:pbe_3_term_weak_form_1} leads to an different weak form: find 
$(u^h,u^r) \in H^1_{u^s}(\Omega_m) \times H^1_g(\Omega)$ such that
\begin{equation}
	\label{eq:pbe_3_term_weak_form_2}
	\begin{cases}
	\begin{aligned}
	&(\epsilon \nabla u^h, \nabla w)_m = 0\\
	&(\epsilon \nabla u^r, \nabla v) + (\okappa^2 N(u^r), v) + (\epsilon \nabla u^h, \nabla v)_m = - \langle \epsilon_m 
	\frac{\partial u^s}{\partial \bn} , v \rangle_\Gamma, 
	\end{aligned}
	\end{cases}
\end{equation}
for all $(w,v) \in H^1_0(\Omega_m) \times H^1_0(\Omega)$. The weak forms \eqref{eq:pbe_3_term_weak_form_1} and \eqref{eq:pbe_3_term_weak_form_2} are mathematically equivalent, however, the form \eqref{eq:pbe_3_term_weak_form_2} is amenable to defining the adjoint operator as  discussed in \S \ref{sec:apost_analysis}.

\subsection{Quantity of interest: solvation free energy}
The QoI may be any bounded linear functional of the{} weak solution $(u^h,u^r)$. An important physical quantity computed from the solution of the PBE is  electrostatic free energy of solvation \cite{GYW2007},
\begin{equation}
\Delta G_{\text{sol}} =  \frac{\alpha}{2} \int \sum_{i=1}^{m} \q_i \delta(\bx - \bx_i) (u^h(x) + u^r(x))\, \mathrm{d}x,
\end{equation}
where $\alpha = k_B T/e_c$. Unfortunately,  $\Delta G_{\text{sol}}$ is not a bounded linear functional in $H^1_0(\Omega)$ due to the presence of $\delta$ functions. A common approach is to ``mollify'' the unbounded functional~\cite{ABC+2011,BangerthRannacher,AO2000} to obtain a bounded linear functional. We thus define our quantity of interest to be a mollified version of solvation free energy, scaled by $2/\alpha$ for simplicity, as
\begin{equation}
\label{eq:qoi}
Q(u^h,u^r) =  \int_{\Omega_m} \sum_{i=1}^{m} \q_i  \eta^{-3} \rho( (x - x_i)/\eta) (u^h(x) + u^r(x))\, \mathrm{d}x = (\psi,u^h+u^r)_m,
\end{equation}
where 
\begin{equation}
\label{eq:psi}
\psi(x) = \sum_{i=1}^{m} \q_i  \eta^{-3} \rho( (x - x_i)/\eta),
\end{equation}
$\rho$ is the standard mollifier
\begin{equation}
\rho(x) = \begin{cases}
c e^{(-1/(1 - |x|^2))} \quad \text{ if } |x| < 1,\\
0 \quad \text{otherwise},
\end{cases}
\end{equation}
$|x|$ denotes the Euclidean norm of $x \in \mathbb{R}^3$ and $c$ is a scaling constant to ensure that $\int_{\mathbb{R}^3} \rho(x) = 1$.  Now, as $\eta \rightarrow 0$, $\eta^{-3} \rho(x/\eta) \rightarrow \delta (x)$. Hence the value of the QoI approaches the value  of the (scaled) solvation free energy for small values of $\eta$.

\subsection{Finite element method}
We discretize $\Omega$ and $\Omega_m$ into three dimensional triangulations $\mathcal{T}$ and $\mathcal{T}_m$. We assume that the interface $\Gamma$ is polygonal and exactly represented by the triangulation. Although the triangulations $\mathcal{T}$ and $\mathcal{T}_m$ may differ in $\Omega_m$, they respect the interface $\Gamma$ in the sense that $ (\cup_{T_m \in \mathcal{T}_m} {T}_m) \cap \Gamma = (\cup_{T \in \mathcal{T}} {T}) \cap \Gamma $. Each of these triangulations is arranged in such a way that the union of the elements of $\mathcal{T}$ (resp. $\mathcal{T}_m$ ) is $\Omega$ (resp. $\Omega_m$) and the intersection of any two elements is either a common edge, node, or is empty.
The finite element space consists of  continuous piecewise polynomials. We let $V_{h}  \subset H_0^1(\Omega)$ (resp. $V_{h}^m  \subset H_0^1(\Omega_m)$) denote the space of continuous piecewise polynomial  functions $v( x) \in \mathbb{R}$ defined on $\mathcal{T}$ (resp. $\mathcal{T}_m)$. Similarly, we let $V_{h,g}$  (resp. $V_{h,u_s}^m)$ be the affine space of continuous piecewise polynomial  functions $v( x) \in \mathbb{R}$ such that $v(x) = g(x)$ for $x$ on $\partial \Omega$ (resp. $v(x) = -u_s(x)$ for $x$ on $\Gamma = \partial \Omega_m$).

The discrete weak problem is to find 
$(U^h,U^r) \in V_{h,u_s}^m \times V_{h,g}$ such that
\begin{equation}
	\label{eq:pbe_3_term_weak_form_2_disc}
	\begin{cases}
	\begin{aligned}
	&(\epsilon \nabla U^h, \nabla w)_m = 0\\
	&(\epsilon \nabla U^r, \nabla v) + (\okappa^2 N(U^r), v) + (\epsilon \nabla U^h, \nabla v)_m = - \langle \epsilon_m 
	\frac{\partial u^s}{\partial \bn} , v \rangle_\Gamma, 
	\end{aligned}
	\end{cases}
\end{equation}
for all $(w,v) \in V_{h}^m \times V_{h}$. Note that throughout this article we use lower case letters for continuum solutions and uppercase letters for discrete solutions.

\section{Adjoint based \emph{a posteriori} analysis}
\label{sec:apost_analysis}
In this section we derive the adjoint equation corresponding to the PBE and then form error representations for both the linearized and nonlinear PBE.

\subsection{Abstract definition of adjoint operator and error representation}
The adjoint operator $\mathcal{L}^\ast : Y^\ast \rightarrow X^\ast$ of a linear operator $\mathcal{L}: X \rightarrow Y$ between Banach spaces $X,Y$  with dual spaces $X^\ast, Y^\ast$ is defined by the bilinear identity~\cite{MAS1996,Kato13,Yosida08},
\begin{equation}
\label{eq:bilin}
 \llangle \mathcal{L} x,y^\ast \rrangle_Y = \llangle x, \mathcal{L}^\ast y^\ast \rrangle_X, \quad x \in X,   y^\ast \in Y^\ast,
\end{equation}
where  $\llangle \cdot, \cdot \rrangle_S$ denotes duality-pairing in the space $S \in \lbrace X, Y \rbrace$. Now,  if $L u = f$ and $L^\ast \phi = \psi$, and $U$ is a discrete approximation to $u$,   we obtain a representation for the error $(u - U)$ as
\begin{equation}
\label{eq:abstract_err_rep}
\llangle \psi, u-U \rrangle_X = \llangle L^\ast \phi, u-U \rrangle_X = \llangle \phi, L u - L U \rrangle_Y = \llangle \phi, f  - LU\rrangle_Y.
\end{equation}
The above abstract error representation is a standard form for all adjoint based error analysis: residual(s) of the discrete solution weighted by the adjoint solution(s). The weighting of the residual by the adjoint solution accounts for the accumulation and cancellation of error in the discrete solution. We remark that the derivation in \eqref{eq:abstract_err_rep} is similar to the derivation of standard Green's functions in PDE analysis, and hence adjoint solutions may be thought of as generalized Green's functions~\cite{EHL05}.

\subsection{\emph{A posteriori} analysis of the linearized PBE}
This section forms an adjoint operator and an error representation for the linearized PBE.

\subsubsection{Adjoint operator for the linearized PBE}
In the context of the linearized PBE, the duality pairing $\llangle \mathcal{L} x,y^\ast \rrangle_Y$ is described by the left hand side of \eqref{eq:pbe_3_term_weak_form_2} with $N(u) = u$. Applying the definition \eqref{eq:bilin}  leads to the following adjoint problem: find $(\phi^h, \phi^r) \in H^1_0(\Omega_m) \times H^1_0(\Omega)$ such that
\begin{equation}
\begin{cases}
	\label{eq:pbe_3_term_adj}
	\begin{aligned}
	&(\epsilon \nabla \phi^r, \nabla v) + (\okappa^2 \phi^r, v)   =  (\psi,v)_m,\\
	&(\epsilon \nabla \phi^h, \nabla w)_m  + (\epsilon \nabla \phi^r, \nabla w)_m =  (\psi, w)_m,
	\end{aligned}
	\end{cases}
\end{equation}
for all $w,v \in H^1_0(\Omega_m) \times H^1_0(\Omega)$. 
Here $\psi$ arises from the definition of the QoI, see \eqref{eq:psi}.
Observe that \eqref{eq:pbe_3_term_adj} is also a one way coupled system, similar to \eqref{eq:pbe_3_term_weak_form_2}, however, the direction of coupling is now reversed: we first solve for the component in $\Omega$ and use that in the equation posed on $\Omega_m$.

\begin{rem}
In Section 4.2 of \cite{ABC+2011},  an adjoint to the three term split PBE is defined as:
find $(w^h, w^r) \in H^1_0(\Omega_m) \times H^1_0(\Omega)$ such that
\begin{equation}
	\label{eq:pbe_3_term_adj_wrong}
	\begin{cases}
	\begin{aligned}
	&(\epsilon \nabla w^r, \nabla v) + (\okappa^2 w^r, v)   =  ({\psi},v)_m,\\
	&(\epsilon \nabla w^h, \nabla w)_m  + \langle  w^h , \epsilon_m 
	\frac{\partial w}{\partial \bn}  \rangle_\Gamma =  ({\psi}, w)_m,
	\end{aligned}
	\end{cases}
\end{equation}
for all $(w,v) \in H^1_0(\Omega_m) \times H^1_0(\Omega)$. However, this is not a well-posed problem as $\langle \epsilon_m 
	\frac{\partial w}{\partial \bn} , w^h \rangle_\Gamma$ is not  continuous in $w$ for all $w \in H^1_0(\Omega_m)$. Continuity of $\langle \epsilon_m 
	\frac{\partial w}{\partial \bn} , w^h \rangle_\Gamma$ requires addition regularity on $w$ e.g. $w \in H^{\frac{3}{2}}(\Omega)$.
\end{rem}

\subsubsection{Error representation for the linearized PBE}

The effect on approximating the boundary  conditions on the interface $\Gamma$ for  $U^h$ may have significant effect on the accuracy of the method. Hence, we quantify the effect of boundary conditions, both at $\Gamma = \partial \Omega_m$ corresponding to the harmonic component $U^h$ and at $\partial \Omega$ corresponding to the regular component $U^r$. 
We employ the decompositions 
\begin{equation}
\label{eq:bdry_decomp_u}
u^r =  u_0^r + {u}_d^r \qquad \text{and} \qquad u^h =  u_0^h + {u}_d^h
\end{equation}
where $u_0^r \in H^1_0(\Omega)$ (resp. $u_0^h \in H^1_0(\Omega_m)$) and ${u}_d^r \in H^1(\Omega)$ (resp. ${u}_d^h \in H^1(\Omega_m)$ ) such that ${u}_d^r  = g$ on $\partial \Omega$ (resp. ${u}_d^h  = -u^s$ on $\partial \Omega_m = \Gamma$) . Similarly we have the decompositions,
 \begin{equation}
 \label{eq:bdry_decomp_U}
 U^r = U^r_0 + {U}_d^r  \qquad \text{and} \qquad  U^h = U^h_0 + {U}_d^h
 \end{equation}
where $U^r_0 \in V_{h}$ (resp. $U^h_0 \in V_{h}^m$) and ${U}_d^r \in V_{h,g}$ (resp. ${U}_d^r \in V_{h,u_s}^m$). Note that due to the finite dimension of $V_{h,g}$ and $V_{h,u_s}^m$
and the nature of the boundary conditions $g$ and $-u^s$, ${u}_d^r \neq {U}_d^r$ and ${u}_d^h \neq {U}_d^h$. Moreover, there are infinitely many choices for the functions $ u_d^h, u_d^r, U_d^h, U_d^h $ and we assume a choice is made such that these functions are known. This leads to the following error representation.

\begin{thm}
\label{thm:err_rep_lin_pbe}
Let $(u^h,u^r)$ be the true solutions to the linearized PBE~\eqref{eq:pbe_3_term_weak_form_2} with $N(u^r) = u^r$, $(U^h,U^r)$ be the finite element solutions to the discrete weak form \eqref{eq:pbe_3_term_weak_form_2_disc} and $(\phi^h,\phi^r)$ be the solutions to the adjoint weak form~\eqref{eq:pbe_3_term_adj}. Then the error in the QoI~\eqref{eq:qoi} is given by
\begin{equation}
\label{eq:err_rep_lin}
Q(u^h-U^h, u^r-U^r) = E^r + E^m + E^\Gamma + E^{\partial \Omega} + E^{\mathrm{neg}},
\end{equation}
where
\begin{equation}
	\begin{aligned}
	&E^r && = 	-\langle \epsilon_m 
	\frac{\partial u^s}{\partial \bn} , \phi^r \rangle_\Gamma - (\epsilon \nabla U^r, \nabla \phi^r) - (\okappa^2 U^r, \phi^r) - (\epsilon \nabla U^h, \nabla \phi^r)_m \\
	&E^m &&= -(\epsilon \nabla U^h, \nabla \phi^r)_m\\
	&E^{\mathrm{neg}} &&= Q( u_d^h - U^h_d, u_d^r - U_d^r ) = (\psi, (u_d^h - U^h_d) +  (u_d^r - U_d^r) )_m\\
	&E^{\partial \Omega} &&= (\epsilon \nabla \phi^r, \nabla (U^r_d - u^r_d)) + (\okappa^2 \phi^r, (U^r_d - u^r_d)) \\
	&E^\Gamma &&=  (\epsilon \nabla \phi^h, \nabla (U^h_d - u^h_d))_m  + (\epsilon \nabla \phi^r, \nabla (U^h_d - u^h_d))_m.
	\end{aligned}
\end{equation}

\end{thm}

\begin{proof}
The tuple $(u^h - U^h, u^r - U^r)$ is not in $H^1_0(\Omega_m) \times H^1_0(\Omega)$. However, if we use the decompositions \eqref{eq:bdry_decomp_u} and \eqref{eq:bdry_decomp_U} along with the linearity of the QoI $Q$,
\begin{equation}
\label{eq:temp_err_0}
\begin{aligned}
	Q(u^h - U^h, u^r - U^r) &=  Q( (u_0^h + U^h_d) - U^h,  (u^r_0  + U^r_d) - U^r  ) + Q( u_d^h - U^h_d, u_d^r - U_d^r )\\
	&=  Q( (u_0^h + U^h_d) - U^h,  (u^r_0  + U^r_d) - U^r ) + E^{\mathrm{neg}}.
	\end{aligned}
\end{equation}
The tuple $((u_0^h + U^h_d) - U^h,  (u^r_0  + U^r_d) - U^r) $ is in $H^1_0(\Omega_m) \times H^1_0(\Omega)$. Hence, setting $w = ((u_0^h + U^h_d) - U^h)$ and $v = ((u^r_0  + U^r_d) - U^r)$ in the adjoint equation \eqref{eq:pbe_3_term_adj} and adding the two equations leads to,

\begin{equation}
\label{eq_temp_err_3}
\begin{aligned}
&Q( (u_0^h + U^h_d) - U^h,  (u^r_0  + U^r_d) - U^r)  )  = (\psi, (u_0^h + U^h_d) - U^h,  (u^r_0  + U^r_d) - U^r )_m\\
&= (\epsilon \nabla \phi^r, \nabla ((u^r_0  + U^r_d) - U^r)) + (\okappa^2 \phi^r, (u^r_0  + U^r_d) - U^r) \\
& \qquad +  (\epsilon \nabla \phi^h, \nabla ((u_0^h + U^h_d) - U^h))_m  + (\epsilon \nabla \phi^r, \nabla ((u_0^h + U^h_d) - U^h))_m.
\end{aligned}
\end{equation}
Substituting $u^r_0 = (u^r_0 +u^r_d) - u^r_d = u^r - u^r_d $ and similarly $u^h_0 = u^h - u^h_d$ and  rearranging,
\begin{equation}
\label{eq:temp_err_1}
\begin{aligned}
&Q( (u_0^h + U^h_d) - U^h,  (u^r_0  + U^r_d) - U^r  ) \\
 &=(\epsilon \nabla \phi^r, \nabla u^r) + (\okappa^2 \phi^r, u^r) + (\epsilon \nabla \phi^h, \nabla u^h)_m  + (\epsilon \nabla \phi^r, \nabla u^h)_m\\
 & -\left((\epsilon \nabla \phi^r, \nabla U^r) + (\okappa^2 \phi^r, U^r) + (\epsilon \nabla \phi^h, \nabla U^h)_m  + (\epsilon \nabla \phi^r, \nabla U^h)_m)\right)\\
 &=(\epsilon \nabla \phi^r, \nabla (U^r_d - u^r_d)) + (\okappa^2 \phi^r, (U^r_d - u^r_d)) + (\epsilon \nabla \phi^h, \nabla (U^h_d - u^h_d))_m  \\
 &+ (\epsilon \nabla \phi^r, \nabla (U^h_d - u^h_d))_m.
\end{aligned}	
\end{equation}
Now, since $(u^h,u^r)$ is the true solution, it satisfies the weak form \eqref{eq:pbe_3_term_weak_form_2}. Substituting this in \eqref{eq:temp_err_1} and rearranging terms,
\begin{equation}
\label{eq:temp_err_2}
\begin{aligned}
&Q( (u_0^h + U^h_d) - U^h,  (u^r_0  + U^r_d) - U^r)   \\
&= - \langle \epsilon_m \frac{\partial u^s}{\partial \bn} , \phi^r \rangle_\Gamma  - (\epsilon \nabla \phi^r, \nabla U^r) - (\okappa^2 \phi^r, U^r)   - (\epsilon \nabla \phi^r, \nabla U^h)_m\\
&\quad -(\epsilon \nabla \phi^h, \nabla U^h)_m + (\epsilon \nabla \phi^r, \nabla (U^r_d - u^r_d)) + (\okappa^2 \phi^r, (U^r_d - u^r_d)) \\
& \quad + (\epsilon \nabla \phi^h, \nabla (U^h_d - u^h_d))_m  + (\epsilon \nabla \phi^r, \nabla (U^h_d - u^h_d))_m
\end{aligned}	
\end{equation}
Combining \eqref{eq:temp_err_0} and \eqref{eq:temp_err_2} completes the proof.

\end{proof}

In the above theorem $E^r$,  $E^m$,  $E^\Gamma$,  $E^{\partial \Omega}$ and $E^{\mathrm{neg}}$ denote different sources of error. The first four terms $E^r$,  $E^m$,  $E^\Gamma$ and $E^{\partial \Omega}$ have the form of adjoint weighted residuals and reflect error contributions due to FEM solution of $u^r$, FEM solution of $u^h$, representation of boundary data for $u^h$ and representation of boundary data for $u^r$. The term $E^{\mathrm{neg}}$, which is computable since all the functions involved are known, is referred to as the ``negligible'' component of error as it is typically negligible due to the standard choice of the boundary functions. 
See \S \ref{sec:num_exps} for more details on the choice of the boundary functions involved as well as the numerical value of this term.

\subsection{\emph{A posteriori} Analysis of the nonlinear PBE}

We now extend the ideas for the linearized PBE to derive an adjoint and error representation for the nonlinear PBE.

\subsubsection{Adjoint operator for the nonlinear PBE}
The extension of the above approach to the nonlinear PBE is complicated by the fact that there is no unique definition of an adjoint operator corresponding to a nonlinear differential operator. Rather, an adjoint problem useful for the purpose at hand has to be selected. A common choice useful for various kinds of analysis  is based on linearization~\cite{MAS1996,Marchuk1995}. Defining $z = su + (1-s)U$ and $\overline{\alpha}(x) =   \int_0^1  \cosh(z(x)) \, \mathrm{d}s$, we observe that
\begin{equation}
\label{eq:adj_non_prop}
\sinh(u^r) - \sinh(U^r) = \int_0^1 \frac{d}{ds} \sinh(z) \, \mathrm{d}s =  \int_0^1  \cosh(z) \, \mathrm{d}s \, \, (u^r - U^r) = \overline{\alpha}(u^r - U^r).
\end{equation}
Then the adjoint corresponding to the nonlinear PBE~\eqref{eq:pbe_3_term_weak_form_2} is: find $(\phi^h, \phi^r) \in$ $\phantom a$ $H^1_0(\Omega_m) \times H^1_0(\Omega)$ such that
\begin{equation}
	\label{eq:pbe_3_term_adj_nonlin}
	\begin{cases}
	\begin{aligned}
	&(\epsilon \nabla \phi^r, \nabla v) + (\okappa^2  \overline{\alpha} \phi^r, v)   =  (\psi,v)_m,\\
	&(\epsilon \nabla \phi^h, \nabla w)_m  + (\epsilon \nabla \phi^r, \nabla w)_m =  (\psi, w)_m,
	\end{aligned}
	\end{cases}
\end{equation}
for all $w,v \in H^1_0(\Omega_m) \times H^1_0(\Omega)$. In practice, we cannot compute the linearization $\overline{\alpha}$ since we do not know the true solution $u^r$. Instead, the differential operator is typically linearized around the numerical solution, in this case $U^r$. The resulting estimate can be shown to converge to the true estimate in the limit of refined discretization \cite{Estep:Larson:00}. In practice, this approach yields robustly accurate error estimates.

\subsubsection{Error representation for the nonlinear PBE}
The above adjoint equation leads to the following error representation for the nonlinear PBE.
\begin{thm}
\label{tmh:non_pbe_err_rep}
Let $(u^h,u^r)$ be the true solutions to the nonlinear PBE~\eqref{eq:pbe_3_term_weak_form_2} with $N(u^r) = \sinh(u^r)$, $(U^h,U^r)$ be the finite element solutions to the discrete weak form \eqref{eq:pbe_3_term_weak_form_2_disc} and $(\phi^h,\phi^r)$ be the solutions to the adjoint weak form~\eqref{eq:pbe_3_term_adj}. Then the error in the QoI~\eqref{eq:qoi} is given by,
\begin{equation}
\label{eq:err_rep_nonlin}
Q(u^h-U^h, u^r-U^r) = E^r + E^m + E^\Gamma + E^{\partial \Omega} + E^{\mathrm{neg}},
\end{equation}
where
\begin{equation}
\label{eq:err_comps_nonlin}
	\begin{aligned}
	&E^r && = 	-\langle \epsilon_m 
	\frac{\partial u^s}{\partial \bn} , \phi^r \rangle_\Gamma - (\epsilon \nabla U^r, \nabla \phi^r) - (\okappa^2 \sinh(U^r), \phi^r) - (\epsilon \nabla U^h, \nabla \phi^r)_m \\
	&E^m &&= -(\epsilon \nabla U^h, \nabla \phi^r)_m\\
	&E^{\mathrm{neg}} &&= Q( (u_d^h - U^h_d, u_d^r - U_d^r) ) = (\psi, (u_d^h - U^h_d) +  (u_d^r - U_d^r) )_m\\
	&E^{\partial \Omega} &&= (\epsilon \nabla \phi^r, \nabla (U^r_d - u^r_d)) + (\okappa^2 \overline{\alpha} \phi^r, (U^r_d - u^r_d)) \\
	&E^\Gamma &&=  (\epsilon \nabla \phi^h, \nabla (U^h_d - u^h_d))_m  + (\epsilon \nabla \phi^r, \nabla (U^h_d - u^h_d))_m
	\end{aligned}
\end{equation}

\end{thm}

\begin{proof}
The proof is similar to the proof of Theorem~\ref{thm:err_rep_lin_pbe}. The difference is in the term $(\okappa^2 \phi^r, (u^r_0  + U^r_d) - U^r))$ in \eqref{eq_temp_err_3} which now becomes,
\begin{equation}
	\begin{aligned}
		&(\okappa^2  \overline{\alpha} \phi^r, (u^r_0  + U^r_d) - U^r))\\
		&= (\okappa^2  \overline{\alpha} \phi^r, (u^r - U^r)) + (\okappa^2  \overline{\alpha} \phi^r, U^r_d - u^r_d))
	\end{aligned}
\end{equation}
where we again made the substitution $u^r_0 = u^r - u^r_d $. Combining the above equation with \eqref{eq:adj_non_prop} leads to
\begin{equation}
\begin{aligned}
		&(\okappa^2  \overline{\alpha} \phi^r, (u^r_0  + U^r_d) - U^r))\\
		&= (\phi^r, \okappa^2  \sinh(u^r)) - (\phi^r, \okappa^2   \sinh(U^r))+ (\okappa^2  \overline{\alpha} \phi^r, U^r_d - u^r_d).
\end{aligned}	
\end{equation}
\end{proof}

\subsubsection{Error representation for the alternate formulation of the PBE}
In this article, the focus is on quantifying the error due to the solution of the FEM problem in \eqref{eq:pbe_3_term_weak_form_2_disc} which corresponds to the solution of \eqref{eq:pbe_3_term_weak_form_2}. However, some existing codes may be based on the discrete solution of the weak form \eqref{eq:pbe_3_term_weak_form_1}.
 In such a case, the error representation is easily modified as shown in the next theorem. 

\begin{thm}
\label{tmh:non_pbe_err_rep_formulation_1}
Let $(u^h,u^r)$ be the true solutions to the (linearized or nonlinear) PBE~\eqref{eq:pbe_3_term_weak_form_2}, $(U^h,U^r)$ be the finite element solutions to the  discrete weak form corresponding to \eqref{eq:pbe_3_term_weak_form_1} and $(\phi^h,\phi^r)$ be the solutions to the adjoint weak form~\eqref{eq:pbe_3_term_adj}. Then the error in the QoI~\eqref{eq:qoi} is given by,
\begin{equation}
\label{eq:err_rep_nonlin_formulation_1}
Q(u^h-U^h, u^r-U^r) = \tilde{E}^r + E^m + E^\Gamma + E^{\partial \Omega} + \tilde{E}^{\mathrm{har}}+ E^{\mathrm{neg}},
\end{equation}
where
\begin{equation}
\label{eq:err_comps_nonlin_formulation_1}
	\begin{aligned}
	&E^r && = 	-\langle \epsilon_m 
	\frac{\partial u^s}{\partial \bn} , \phi^r \rangle_\Gamma - (\epsilon \nabla U^r, \nabla \phi^r) - (\okappa^2 N(U^r), \phi^r) - \langle \epsilon_m 
	\frac{\partial U^h}{\partial \bn}, \phi^r \rangle_\Gamma  \\
	&E^{\mathrm{har}} &&= \langle \epsilon_m	\frac{\partial U^h}{\partial \bn}, \phi^r \rangle_\Gamma - (\epsilon \nabla U^h, \nabla \phi^r)_m
	\end{aligned}
\end{equation}
and the remaining terms are the same as in Theorem \ref{tmh:non_pbe_err_rep}.
\end{thm}

\begin{proof}
Adding and subtracting $\langle \epsilon_m	\frac{\partial U^h}{\partial \bn}, \phi^r \rangle_\Gamma$ to term $E^r$ in Theorem \ref{tmh:non_pbe_err_rep} completes the proof.
\end{proof}

\section{Refinement strategies based on \emph{a posteriori} error estimates}
\label{sec:refine_strategies_apost_ideas}

This section discusses the accuracy of \emph{a posteriori} error estimates and the potential for obtaining accurate QoI values using the error information to refine the discretization.

\subsection{\emph{A posteriori} error estimates: implementation and accuracy}
The error representations \eqref{eq:err_rep_lin} and \eqref{eq:err_rep_nonlin} involve analytic adjoint solutions $(\phi^h,\phi^r)$ and representation of boundary condition{}s by functions $(u^h_d, u^r_d)$. In practice, these quantities need to be estimated computationally. 
As is common in literature for adjoint based \emph{a posteriori} analysis, the adjoint solutions are approximated in a space $W^h$ (resp. $W^r$) such that $V^h \subset W^h$ (resp. $V^r \subset W^r$)~\cite{Estep:Larson:00,estep_sinum_95,eehj_actanum_95,CET+2016,CEG+2015,Estep:Larson:00,estep_sinum_95,eehj_actanum_95,CET+2016,CEG+2015,CBH+2014,CET+2016,CET-09,barth04}. 
$W^h$ may be obtained by  refining the mesh or by increasing the polynomial order. 
Similarly, the functions $(u^h_d, u^r_d)$ are approximated in $W^h$, such that they satisfy the  boundary condition exactly on a boundary vertex  and are zero on the interior vertices. 
These approximations lead to error estimates from the error representations \eqref{eq:err_rep_lin} and \eqref{eq:err_rep_nonlin}. Since the formulas are similar, except that the $(\phi^h,\phi^r)$ and $(u^h_d, u^r_d)$  are replaced by their approximations, we avoid re-writing the error estimates explicitly and instead now refer to \eqref{eq:err_rep_lin} and \eqref{eq:err_rep_nonlin}  as error estimates. 

The accuracy of the error estimate is measured by the effectivity ratio defined as
\begin{equation*}
\gamma_{\rm eff} = \frac{\mbox{Estimated error}}{\mbox{True error}} \, .
\end{equation*}
An accurate error estimator has an effectivity ratio close to one. Since the true solution is not known, we compute a more accurate reference numerical solution using a higher dimensional space for measuring the true error.

\subsection{Guiding refinement decisions using error estimates}

\subsubsection{Error Contributions and cancellation of error}
Once the error estimate is in place, its various components $E^r, E^m,  E^\Gamma$ and  $E^{\partial \Omega}$  reflect different sources of error. Refinement strategies based on these components can then be derived. For example, if  $E^\Gamma $ is the dominant component, then simplices in $ \mathcal{T}_m$ which intersect with the interface $\Gamma$ may be refined to reduce the error. This strategy of refining the mesh is quite different from classical adaptive refinement schemes. One main difference is that, in refining the simplices on the interface to reduce $E^\Gamma$ we may use either uniform refinement or an adaptive refinement strategy. The other difference is in the treatment of cancellation of errors which we now discuss.

Classical adaptive refinement schemes form  elemental error indicators and refine elements which have the largest value of such indicators~\cite{AO2000,NA2007,ETW2010}. 
While adaptive refinement often outperforms uniform refinement, its efficiency is somewhat limited for decreasing error in a QoI as the error contributions may be both positive or negative, and hence there is often significant cancellation of error~\cite{CET+2016}. This is in contrast to reducing error in standard norms which are always positive~\cite{Hols01,BHW00,BCCO09}.  In classical adjoint based adaptivity, the absolute value of the elemental error indicators is taken and the principle of equidistribution applied. By taking the absolute value of the elemental error contributions, the cancellation of error due to opposing signs is lost. This phenomenon, along with a novel refinement strategy based on ``mesoscale'' regions is illustrated for ODEs in \cite{CET+2016}. On the other hand, uniform refinement reduces the error predictably in the asymptotic regime and hence it is expected to reduce
 both the positive and negative elemental contributions equally. Thus, uniform refinement is expected to preserve the cancellation of error and this was observed experimentally in \cite{CET+2016}. Uniform refinement is also more predictable in the expected decrease of error. In this article, we  outline refinement strategies targeting sources of error as well as those based on elemental error indicators.

The main idea behind targeting sources of error to obtain accurate solutions is to  reduce the dominant (in magnitude) source of error $E^r, E^m,  E^\Gamma$ and  $E^{\partial \Omega}$. This is accomplished by refining (either uniformly or adaptively) the corresponding discretization as shown in Table~\ref{tab:dom_disc_refin}.

\begin{table}[!ht]
\centering
\begin{tabular}{c|c}
\toprule
Dominant source & Discretization to refine \\
\midrule
 $E^r$ & Refine $\mathcal{T}$\\
 $E^m$ & Refine $\mathcal{T}_m$\\
 $E^\Gamma$ & Refine simplices containing $\Gamma \cap \Omega_m$\\
 $E^{\partial \Omega}$ & Refine simplices containing $\partial \Omega \cap \Omega$\\
\bottomrule
\end{tabular}
\caption{The discretizations to be refined based on the dominant source of error. The refinement may be uniform or adaptive.}
\label{tab:dom_disc_refin}
\end{table}

\subsubsection{Uniform Contribution Refinement}

In the Uniform Contribution Refinement, we choose the dominant component for refinement if it is at least 3 times larger than the next dominant component, or if both the top two dominant components have the same sign, so that  the cancellation of error is preserved. If this requirement is not satisfied, the scheme defaults to standard uniform refinement.

\subsubsection{Adaptive Contribution Refinement}
The Adaptive Contribution Refinement is similar to the standard algorithms for adjoint weighted adaptive algorithms~\cite{EHL05,barth04,ABC+2011}. E.g., if the aim is to reduce the component $E^r$, then we define an elemental error indicator based on \eqref{eq:err_comps_nonlin} as
\begin{equation}
\eta_T  = |-\langle \epsilon_m 
	\frac{\partial u^s}{\partial \bn} , \phi^r \rangle_{T,\Gamma} - (\epsilon \nabla U^r, \nabla \phi^r)_T - (\okappa^2 \sinh{U}^r, \phi^r)_T - (\epsilon \nabla U^h, \nabla \phi^r)_{T,m} |
\end{equation}
where $T \in \mathcal{T}$, and  the subscripts ${T,\Gamma} $, $T$ and $T,m$ refer to evaluations of the integrals restricted to the element $T$ such that $T \cap \Gamma \neq \phi$, $T \in \Omega$ and $T\cap \Omega_m \neq \phi$ respectively. Once  a per elemental error estimator is defined, the \textit{D\"orfler} scheme is used for marking the elements for refinement~\cite{Do96}. To preserve the cancellation of errors between different sources of error, all sources which have a total error contribution of at least half the dominant error contribution are selected to be adaptively refined.

 \subsubsection{Classical Refinement}
 In the classical adaptive refinement strategy we add up the terms in $E^r, E^m, E^\Gamma, E^{\partial \Omega}$ and $E^\mathrm{neg}$ in Theorem~\ref{tmh:non_pbe_err_rep} so that the error in the QoI for the nonlinear PBE is, $Q(u^h-U^h, u^r-U^r) \equiv E$ is,
\begin{equation}
\label{eq:err_combined}
\begin{aligned}
 E &= -\langle \epsilon_m 
	\frac{\partial u^s}{\partial \bn} , \phi^r \rangle_{\Gamma} - (\epsilon \nabla U^r, \nabla \phi^r) - (\okappa^2 \sinh({U}^r), \phi^r) - (\epsilon \nabla U^h, \nabla \phi^r)_{m} \\
	&-(\epsilon \nabla U^h, \nabla \phi^m)_{m} + (\epsilon \nabla \phi^r, \nabla (U^r_d - u^r_d)) + (\okappa^2 \overline{\alpha} \phi^r, (U^r_d - u^r_d)) \\
	&+(\epsilon \nabla \phi^h, \nabla (U^h_d - u^h_d))_m  + (\epsilon \nabla \phi^r, \nabla (U^h_d - u^h_d))_{m} + (\psi, (u_d^h - U^h_d) +  (u_d^r - U_d^r) )_m
\end{aligned}
\end{equation}
We define  projection operators, $\pi_m: H^1_0(\Omega_m) \rightarrow  V_{h}^m$ and $\pi_r: H^1_0(\Omega) \rightarrow V_{h}$. From \eqref{eq:pbe_3_term_weak_form_2_disc} we have,
\begin{equation}
	\label{eq:galerkin_orthogonality}
	\begin{cases}
	\begin{aligned}
	&(\epsilon \nabla U^h, \nabla \pi_m \phi^m)_m = 0\\
	&(\epsilon \nabla U^r, \nabla \pi_r \phi^r) + (\okappa^2 \sinh(U^r), \pi_r \phi^r) + (\epsilon \nabla U^h, \nabla \pi_r \phi^r)_m = - \langle \epsilon_m 
	\frac{\partial u^s}{\partial \bn} , \pi_r \phi^r \rangle_\Gamma.
	\end{aligned}
	\end{cases}
\end{equation}
Combining \eqref{eq:err_combined} with \eqref{eq:galerkin_orthogonality} leads to the following elemental error indicator for element $T$
\begin{equation}
\label{eq:dwr_indicator}
\begin{aligned}
 \eta_T &= | -\langle \epsilon_m
	\frac{\partial u^s}{\partial \bn} , (\phi^r - \pi_r \phi^r) \rangle_{T, \Gamma} - (\epsilon \nabla U^r, \nabla (\phi^r - \pi_r \phi^r))_T \\
	&- (\okappa^2 \sinh({U}^r), \phi^r - \pi_r \phi^r)_T - (\epsilon \nabla U^h, \nabla (\phi^r - \pi_r \phi^r))_{T,m} \\
	&-(\epsilon \nabla U^h, \nabla (\phi^m - \pi_m \phi^m))_{T,m} + (\epsilon \nabla \phi^r, \nabla (U^r_d - u^r_d))_T + (\okappa^2 \overline{\alpha} \phi^r, (U^r_d - u^r_d))_T \\
	&+(\epsilon \nabla \phi^h, \nabla (U^h_d - u^h_d))_{T,m}  + (\epsilon \nabla \phi^r, \nabla (U^h_d - u^h_d))_{T,m} \\
	&+ (\psi, (u_d^h - U^h_d)_T +  (u_d^r - U_d^r) )_{T,m} |
\end{aligned}
\end{equation}
The elemental error indicator for the linearized PBE is similar except that $\sinh({U}^r)$ is replaced by $U^r$ and $\overline{\alpha}$ by $1$.

\section{Numerical experiments}
\label{sec:num_exps}

We show the accuracy of the {a posteriori} error estimates and utilization of the different sources of error to obtain an accurate computation of the QoI for the Born ion and methanol. The values of the constants in the PBE are chosen as $\epsilon_m=1$, $\epsilon_s = 78$ and $\okappa^2 = 0.918168$ unless otherwise stated.
The value $\okappa^2 = 0.918168$ corresponds to an ionic concentration of 0.1 M. These values reflect typical scenarios for PBE simulations~\cite{BCCO09,CBO12}. 
The initial meshes, defining the domains $\Omega_m$, $\Omega_s$ and the interface $\Gamma$ are generated using GAMer\cite{YHCM08}.
We use the standard space of continuous piecewise linear polynomials for the solution spaces corresponding to $u^m$ and $u^r$, that is for spaces $V_{h}^m$ and $V_{h}$. The spaces for the adjoint solutions $W^h$ and $W^r$ are chosen to be continuous piecewise quadratic polynomials. For ease of implementation, we always ensure that $\mathcal{T}_m = \mathcal{T} \cap \Omega_m$. The QoI \eqref{eq:qoi} requires accurate integration near the points $x_i$. This is achieved by refining the cells near $x_i$ a few times. 
The functions $u^h_d, u^r_d, U^h_d, U^r_d, $ are such that they satisfy the  boundary condition exactly on a boundary vertex  and are zero on the interior vertices. This choice results in the component $E^{\mathrm{neg}}$ being exactly zero which was also verified numerically. 
Experiments are performed for the Born ion and the methanol molecule. The  reference solutions needed for the effectivity ratios are computed using a mesh with 411635 vertices for the Born ion and a mesh with 90264  vertices for the methanol molecule and using continuous piecewise quadratic polynomials for the finite element space. 
The reference values of the QoI for the linearized PBE for the Born ion and methanol are -276.749875 and -48.477443 respectively. The corresponding values for the nonlinear PBE  are -276.825527 and -48.479878.
Since the reference solutions themselves have some error, effectivity ratios are only shown for experiments for which the reference solution is relatively accurate. 
All computations are carried out in the finite element software package DOLFIN from the FEniCS library \cite{LoggWells2010a,LoggWellsEtAl2012a,AlnaesBlechta2015a,LoggMardalEtAl2012a}. The value of $\eta$ for the QoI in \eqref{eq:qoi} was chosen as 0.005.
The \textit{D\"orfler} marking parameter is chosen as 0.2. The projection operators $\pi_m$ and $\pi_r$ were chosen as $L_2$ projectors.

\subsection{Born Ion}

The Born ion consists of a single point charge $Q_1$ in the center of a spherical solute domain $\Omega_m$ of radius $R$~\cite{Kirk34}. The solute is surrounded by a large spherical solvent domain, $\Omega_s$, as depicted in Fig. \ref{fig:born_domains} which has been adopted from \cite{CBO12}.
\begin{figure}[!ht]
\centering
\subfloat[Born ion]{
\label{fig:born_domains}
 \includegraphics[width=.25\textwidth]{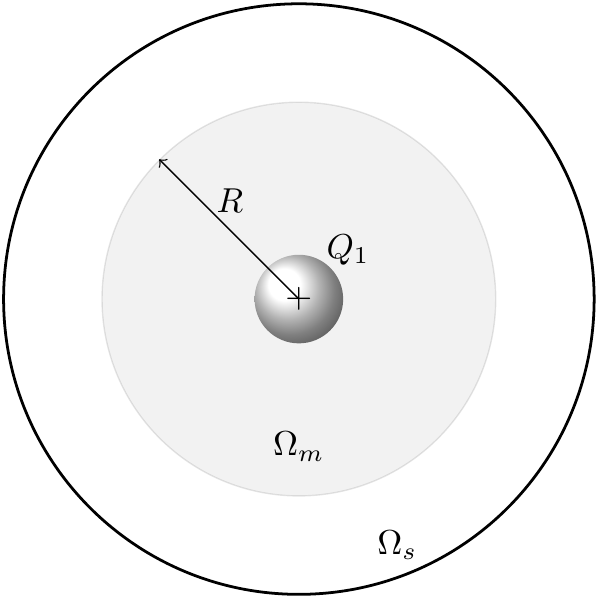}

}
 \hspace{30ex}
 \subfloat[Methanol]{
 \includegraphics[width=.14\textwidth]{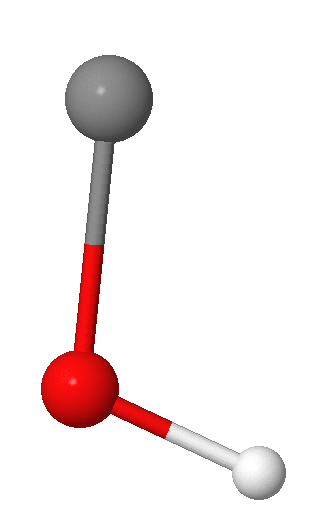}
 \label{fig:methanol} 
 }
\caption{Born ion and Methanol}
\end{figure}
Table~\ref{tab:born_lin_1} shows the error estimate, the effectivity ratio and different sources of error for the linearized PBE for two different meshes: an initial mesh of 6718 vertices and a uniformly refined mesh of 52014 vertices. In both cases, the effectivity ratio is close to one, indicating the accuracy of the error estimate. Moreover, we see that uniform refinement decreases all sources of error while preserving their signs, and hence accounts for cancellation of error. Similar results for the nonlinear PBE are shown in Table~\ref{tab:born_nonlin_1}.

\begin{table}[!ht]
\centering
\begin{tabular}{c|c|c|c|c|c|c}
\toprule
$N$ & Est. Err. & $\gamma_{\rm eff}$ &  $E^r$ & $E^m$ & $E^\Gamma$   & $E^{\partial \Omega}$ \\
\midrule
6718&	-1.14&	  1.05&	2.05e-01&	5.26e-09&	-1.34e+00&	4.86e-04\\
52014&	-0.248&	1.04&	1.19e-01&	2.04e-06&	-3.67e-01&	1.36e-04\\
\bottomrule
\end{tabular}
\caption{Born ion: Error estimate, effectivity ratio and error contributions for the linearized PBE~\eqref{eq:pbe_3_term_weak_form_2} with $N(u) = u$. $N$ is the number of vertices in $\mathcal{T}$. The terms $E^r$, $E^m$, $E^\Gamma$ and $E^{\partial \Omega}$ are defined in Theorem \ref{thm:err_rep_lin_pbe}.  }
\label{tab:born_lin_1}
\end{table}


\begin{table}[!ht]
\centering
\begin{tabular}{c|c|c|c|c|c|c}
\toprule
$N$ & Est. Err. & $\gamma_{\rm eff}$ &  $E^r$ & $E^m$ & $E^\Gamma$   & $E^{\partial \Omega}$ \\
\midrule
6718&	-1.16&	  1.05&	1.88e-01&	5.13e-09&	-1.34e+00&	4.24e-04\\
52014&	-0.254&	1.04&	1.13e-01&	2.69e-06&	-3.67e-01&	1.18e-04\\
\bottomrule
\end{tabular}
\caption{Born ion: Error estimate, effectivity ratio and error contributions for the nonlinear PBE~\eqref{eq:pbe_3_term_weak_form_2} with $N(u) = \sinh(u)$. The terms $E^r$, $E^m$, $E^\Gamma$ and $E^{\partial \Omega}$ are defined in Theorem \ref{tmh:non_pbe_err_rep}.  }
\label{tab:born_nonlin_1}
\end{table}

\subsection{Methanol}
We examine the accuracy of the error estimates in the more challenging setting of a methanol molecule,
obtained from the APBS software package~\cite{BSJ+2001}. The methanol molecule
consists of three charged particles representing charge groups: $\text{CH}_3$ and H with 
positive charges of 0.27 and 0.43 respectively, and an O atom with a negative 
charge of 0.7. The model  is depicted in Fig. \ref{fig:methanol} adopted from~\cite{Chaudhry2011}. The numerical experiments are performed on two meshes: an initial mesh of 11769 vertices and a uniformly refined mesh of 90264 vertices. The results for the linearized and nonlinear PBE are shown in Tables~\ref{tab:meth_lin_1} and \ref{tab:meth_nonlin_1}. The effectivity ratios are again close to $1.0$ and highlight the accuracy and robustness of the error estimate for both cases.

\begin{table}[!ht]
\centering
\begin{tabular}{c|c|c|c|c|c|c}
\toprule
$N$ & Est. Err. & $\gamma_{\rm eff}$ &  $E^r$ & $E^m$ & $E^\Gamma$   & $E^{\partial \Omega}$ \\
\midrule
11769&	-0.924&	  1.03&	4.27e-01&	4.67e-06&	-1.35e+00&	2.43e-06\\
90264&	-0.231&	  1.01&	1.49e-01&	9.34e-06&	-3.81e-01&	6.87e-07\\
\bottomrule
\end{tabular}
\caption{Methanol: Error estimate, effectivity ratio and error contributions for the linearized PBE~\eqref{eq:pbe_3_term_weak_form_2} with $N(u) = u$. }
\label{tab:meth_lin_1}
\end{table}


\begin{table}[!ht]
\centering
\begin{tabular}{c|c|c|c|c|c|c}
\toprule
$N$ & Est. Err. & $\gamma_{\rm eff}$ &  $E^r$ & $E^m$ & $E^\Gamma$   & $E^{\partial \Omega}$ \\
\midrule
11769&	-0.924&	  1.02&	4.27e-01&	3.31e-06&	-1.35e+00&	2.41e-06\\
90264&	-0.232&	  1.01&	1.49e-01&	8.77e-06&	-3.81e-01&	6.81e-07\\
\bottomrule
\end{tabular}
\caption{Methanol: Error estimate, effectivity ratio and error contributions for the nonlinear PBE~\eqref{eq:pbe_3_term_weak_form_2} with $N(u) = \sinh(u)$. }
\label{tab:meth_nonlin_1}
\end{table}

\subsection{Refinement strategies}

We use the different sources of error identified by $E^r, E^m,  E^\Gamma$ and  $E^{\partial \Omega}$  to guide refinement decisions. We first give an example of the effect of refining different discretization components, highlighting the significance of cancellation of error. Finally we present examples based on the Uniform Contribution Refinement, Adaptive Contribution Refinement and Classical Refinement schemes explained in \S \ref{sec:refine_strategies_apost_ideas}.

\subsubsection{Effect of refinement decisions on the QoI error}
Consider the error information in Table \ref{tab:born_lin_1} for the coarse mesh with 6718 vertices and error of $-1.14$. The uniformly refined mesh had 52014 vertices and an error of $-0.248$. On examining the different sources of error, we observe that the dominant error contribution is represented by $E^\Gamma$. Thus, instead of uniformly refining the mesh, we only refine simplices on the interface $\Gamma$. This refinement strategy is carried out by marking simplices in $\Omega_m$ which have one face on the interface, that is, marking simplices $T \in \mathcal{T}$  (and in $\mathcal{T}_m$) such that $T \in \Omega_m$ and $T \cap \Gamma$ is not the empty set. 

The refinement results are shown in Table \ref{tab:err_contrib_example_born}. The ``It.'' indicates the refinement level or iteration, with $0$ indicating the starting coarse mesh. 
After the interface is refined, we arrive at level 1,  corresponding to row having $N = 15541$ vertices. Comparing this to the solution obtained by uniform refinement in Table \ref{tab:born_lin_1}, the error is now slightly less while the number of vertices is only 30\% of the number of vertices of the uniformly refined mesh. This reflects a significant cost savings in obtaining accurate solutions.

At refinement level 1,  once again, $E^\Gamma$  is the dominant component and we again refine the interface to reduce the error to $0.0124$. Now, both $E^r$ and $E^\Gamma$ have the same order of magnitude, but opposite signs. If we still carry on refining the interface to arrive at level 3-$\Gamma$. However, now the error has increased to $0.0621$! This behavior is quite common in numerical simulations, where refining a discretization parameter leads to an increase in the error rather than a reduction. Without the aid of adjoint based estimates, the cause of this increase may be hard to diagnose. The error information at levels 2 and 3 indicate why this increase occurred. The error at level 2 involved cancellation between the terms $E^r$ and $E^\Gamma$. Refining the interface significantly reduced $E^\Gamma$, while having only a marginal effect on $E^r$. Thus, there is less cancellation of error and the error increased to $0.0621$. A better option here is uniform refinement, which preserves the cancellation of error between different contributions~\cite{CET+2016}. The results of applying uniform refinement to Level 2 are shown as level 3-Uniform. The cancellation of error is preserved and the error decreased. Note that both the contribution of $E^\Gamma$ is the same for both levels $3 (\Gamma)$ and $3-$Uniform, while the contribution of $E^r$ only sees a significant decrease at level $3-$Uniform.

\begin{table}[!ht]
\centering
\begin{tabular}{c|c|c|c|c|c|c}
\toprule
It. &$N$ &  Error &  $E^r$ & $E^m$ & $E^\Gamma$   & $E^{\partial \Omega}$ \\
\midrule
0&	6718&	-1.14&	2.05e-01&	5.26e-09&	-1.34e+00&	4.86e-04\\
1&	15541&	-0.214&	1.53e-01&	9.75e-07&	-3.68e-01&	4.78e-04\\
2&	41760&	0.0124&	1.03e-01&	2.30e-07&	-9.07e-02&	4.78e-04\\
3 ($\Gamma$)&	141855&	0.0621&	8.41e-02&	1.38e-07&	-2.25e-02&	4.78e-04\\
3-Uniform&	323084&	0.00853&	3.09e-02&	1.29e-07&	-2.25e-02&	1.35e-04\\
\bottomrule
\end{tabular}
\caption{Born ion:  Error contributions and refinement for the linearized PBE. }
\label{tab:err_contrib_example_born}
\end{table}

\subsubsection{Results for Uniform Contribution Refinement}

The results of the Uniform Contribution Refinement strategy defined in \S \ref{sec:refine_strategies_apost_ideas} for the solution of the linear and nonlinear PBE for the Born ion are shown in Tables~\ref{tab:uni_comp_born_lin} and ~\ref{tab:uni_comp_born_nonlin}, while the results for methanol are shown in Tables~\ref{tab:uni_comp_meth_lin} and \ref{tab:uni_comp_meth_nonlin}. Comparing these results to Tables~\ref{tab:born_lin_1}, \ref{tab:born_nonlin_1}, \ref{tab:meth_lin_1} and \ref{tab:meth_nonlin_1}, we observe that the Uniform Contribution Refinement achieves significantly more accurate solutions with a lower computational cost (as measured by the number of vertices in the mesh) for both Born ion and methanol. An interesting observation is at level 2 of Table~\ref{tab:uni_comp_born_nonlin} where $E^r$ and $E^\Gamma$ have almost the same magnitude but opposite signs. These two sources of error cancel, leading to an unexpectedly low error.

\begin{table}[!ht]
\centering
\begin{tabular}{c|c|c|c|c|c|c|c}
\toprule
It. &$N$ & Est. Err. & $\gamma_{\rm eff}$ &  $E^r$ & $E^m$ & $E^\Gamma$   & $E^{\partial \Omega}$ \\
\midrule
0&	6718&	-1.14&	1.05&	2.05e-01&	5.26e-09&	-1.34e+00&	4.86e-04\\
1&	15541&	-0.214&	1.06&	1.53e-01&	9.75e-07&	-3.68e-01&	4.78e-04\\
2&	41760&	0.0124&	--&	1.03e-01&	2.30e-07&	-9.07e-02&	4.78e-04\\
\bottomrule
\end{tabular}
\caption{Uniform Contribution Refinement strategy defined in \S \ref{sec:refine_strategies_apost_ideas} applied to the linearized PBE for the Born ion. It. refers to the refinement iteration or level.}
\label{tab:uni_comp_born_lin}
\end{table}

\begin{table}[!ht]
\centering
\begin{tabular}{c|c|c|c|c|c|c|c}
\toprule
It. &$N$ & Est. Err. & $\gamma_{\rm eff}$ &  $E^r$ & $E^m$ & $E^\Gamma$   & $E^{\partial \Omega}$ \\
\midrule
0&	6718&	-1.16&	1.05&	1.88e-01&	5.13e-09&	-1.34e+00&	4.29e-04\\
1&	15541&	-0.224&	1.05&	1.44e-01&	5.23e-07&	-3.68e-01&	4.17e-04\\
2&	41760&	3.49e-3&	- &	9.38e-02&	4.86e-07&	-9.07e-02&	4.16e-04\\
\bottomrule
\end{tabular}
\caption{Uniform Contribution Refinement strategy defined in \S \ref{sec:refine_strategies_apost_ideas} applied to the nonlinear PBE for the Born ion. }
\label{tab:uni_comp_born_nonlin}
\end{table}

\begin{table}[!ht]
\centering
\begin{tabular}{c|c|c|c|c|c|c|c}
\toprule
It. &$N$ & Est. Err. & $\gamma_{\rm eff}$ &  $E^r$ & $E^m$ & $E^\Gamma$   & $E^{\partial \Omega}$ \\
\midrule
0&	11769&	-0.924&	1.03&	4.27e-01&	4.67e-06&	-1.35e+00&	2.43e-06\\
1&	20283&	-0.275&	1.02&	1.06e-01&	3.07e-07&	-3.82e-01&	2.40e-06\\
2&	46212&	-0.094&	0.998&	1.20e-03&	7.72e-07&	-9.52e-02&	2.40e-06\\
\bottomrule
\end{tabular}
\caption{Uniform Contribution Refinement strategy defined in \S \ref{sec:refine_strategies_apost_ideas} applied to the linearized PBE for the Methanol. }
\label{tab:uni_comp_meth_lin}

\end{table}

\begin{table}[!ht]
\centering
\begin{tabular}{c|c|c|c|c|c|c|c}
\toprule
It. &$N$ & Est. Err. & $\gamma_{\rm eff}$ &  $E^r$ & $E^m$ & $E^\Gamma$   & $E^{\partial \Omega}$ \\
\midrule
0&	11769&	-0.924&	1.03&	4.27e-01&	3.31e-06&	-1.35e+00&	2.41e-06\\
1&	20283&	-0.276&	1.02&	1.06e-01&	8.55e-07&	-3.82e-01&	2.38e-06\\
2&	46212&	-0.0944&	0.996&	8.65e-04&	4.38e-07&	-9.52e-02&	2.38e-06\\
\bottomrule
\end{tabular}
\caption{Uniform Contribution Refinement strategy defined in \S \ref{sec:refine_strategies_apost_ideas} applied to the nonlinear PBE for Methanol. }
\label{tab:uni_comp_meth_nonlin}
\end{table}

\subsubsection{Results for Adaptive Contribution Refinement}
The results of the Adaptive Contribution Refinement strategy defined in \S \ref{sec:refine_strategies_apost_ideas} for the solution of the linear and nonlinear PBE for the Born ion are shown in Tables~\ref{tab:adapt_comp_born_lin} and \ref{tab:adapt_comp_born_nonlin}, while the results for methanol are shown in Tables~\ref{tab:adapt_comp_meth_lin} and \ref{tab:adapt_comp_meth_nonlin}. Comparing these results to Tables~\ref{tab:born_lin_1}, \ref{tab:born_nonlin_1},  \ref{tab:meth_lin_1} and \ref{tab:meth_nonlin_1}, we observe that the Adaptive Contribution Refinement is almost an order of magnitude more accurate for a uniformly refined mesh having the same number of vertices.
 Adaptive Contribution Refinement also outperforms the Uniform Contribution Refinement strategy  for relatively small values of $N$. A couple of interesting observations are in order. In Table~\ref{tab:adapt_comp_born_nonlin} the error decreases up to level 6 after which loss of error cancellation  leads to an increase on level 7. In Table~\ref{tab:adapt_comp_meth_nonlin} we have an unexpectedly low error due to the cancellation between the terms $E^r$ and $E^\Gamma$.

\begin{table}[!ht]
\centering
\begin{tabular}{c|c|c|c|c|c|c|c}
\toprule
It. &$N$ & Est. Err. & $\gamma_{\rm eff}$ &  $E^r$ & $E^m$ & $E^\Gamma$   & $E^{\partial \Omega}$ \\
\midrule
0&	6718&	-1.14&	1.05&	2.05e-01&	5.26e-09&	-1.34e+00&	4.86e-04\\
1&	9476&	-1.07&	1.03&	1.50e-01&	1.07e-07&	-1.22e+00&	4.79e-04\\
2&	14273&	-0.521&	1.02&	1.24e-01&	3.08e-07&	-6.45e-01&	4.78e-04\\
3&	21293&	-0.235&	1.02&	1.07e-01&	3.88e-07&	-3.42e-01&	4.78e-04\\
4&	33320&	-0.0895&	1.02&	9.68e-02&	2.72e-07&	-1.87e-01&	4.78e-04\\
5&	60908&	-0.00803&	0.98&	8.22e-02&	2.61e-07&	-9.07e-02&	4.77e-04\\
6&	112597&	0.0112&	1.13&	5.90e-02&	1.75e-07&	-4.83e-02&	4.77e-04\\
7&	206897&	0.0183&	1.11&	4.31e-02&	2.87e-07&	-2.53e-02&	4.75e-04\\
\bottomrule
\end{tabular}
\caption{Adaptive Contribution Refinement strategy defined in \S \ref{sec:refine_strategies_apost_ideas} applied to the linearized PBE for the Born ion. }
\label{tab:adapt_comp_born_lin}
\end{table}

\begin{table}[!ht]
\centering
\begin{tabular}{c|c|c|c|c|c|c|c}
\toprule
It. &$N$ & Est. Err. & $\gamma_{\rm eff}$ &  $E^r$ & $E^m$ & $E^\Gamma$   & $E^{\partial \Omega}$ \\
\midrule
0&	6718&	-1.16&	1.05&	1.88e-01&	5.13e-09&	-1.34e+00&	4.29e-04\\
1&	9476&	-1.08&	1.02&	1.38e-01&	8.05e-08&	-1.22e+00&	4.19e-04\\
2&	14273&	-0.531&	1.02&	1.14e-01&	4.24e-07&	-6.45e-01&	4.17e-04\\
3&	21293&	-0.244&	1.02&	9.77e-02&	7.15e-08&	-3.42e-01&	4.16e-04\\
4&	33320&	-0.0983&	1.01&	8.80e-02&	3.05e-07&	-1.87e-01&	4.16e-04\\
5&	54112&	-0.021&	0.98&	7.94e-02&	-7.17e-08&	-1.01e-01&	4.16e-04\\
6&	106120&	-0.00159&	-&	4.96e-02&	3.95e-08&	-5.16e-02&	4.13e-04\\
7&	200323&	0.0109&	1.22&	3.66e-02&	1.31e-07&	-2.61e-02&	4.10e-04\\
\bottomrule
\end{tabular}
\caption{Adaptive Contribution Refinement strategy defined in \S \ref{sec:refine_strategies_apost_ideas} applied to the nonlinear PBE for the Born ion. }
\label{tab:adapt_comp_born_nonlin}
\end{table}

\begin{table}[!ht]
\centering
\begin{tabular}{c|c|c|c|c|c|c|c}
\toprule
It. &$N$ & Est. Err. & $\gamma_{\rm eff}$ &  $E^r$ & $E^m$ & $E^\Gamma$   & $E^{\partial \Omega}$ \\
\midrule
0&	11769&	-0.924&	1.03&	4.27e-01&	4.67e-06&	-1.35e+00&	2.43e-06\\
1&	12380&	-0.564&	1.04&	1.98e-01&	1.55e-06&	-7.62e-01&	2.42e-06\\
2&	13852&	-0.391&	1.03&	8.72e-02&	1.56e-06&	-4.78e-01&	2.41e-06\\
3&	17353&	-0.275&	1.02&	1.21e-02&	-1.17e-06&	-2.87e-01&	2.40e-06\\
4&	22732&	-0.206&	1.01&	-2.94e-02&	-1.81e-06&	-1.77e-01&	2.40e-06\\
5&	33019&	-0.156&	0.99&	-5.16e-02&	-5.96e-07&	-1.04e-01&	2.40e-06\\
6&	50784&	-0.127&	0.98&	-6.92e-02&	-1.08e-06&	-5.80e-02&	2.40e-06\\
7&	86224&	0.000197&	-&	3.48e-02&	2.20e-07&	-3.46e-02&	2.40e-06\\
\bottomrule
\end{tabular}
\caption{Adaptive Contribution Refinement strategy defined in \S \ref{sec:refine_strategies_apost_ideas} applied to the linearized PBE for Methanol. }
\label{tab:adapt_comp_meth_lin}
\end{table}

\begin{table}[!ht]
\centering
\begin{tabular}{c|c|c|c|c|c|c|c}
\toprule
It. &$N$ & Est. Err. & $\gamma_{\rm eff}$ &  $E^r$ & $E^m$ & $E^\Gamma$   & $E^{\partial \Omega}$ \\
\midrule
0&	11769&	-0.924&	1.03&	4.27e-01&	3.31e-06&	-1.35e+00&	2.41e-06\\
1&	12380&	-0.564&	1.04&	1.97e-01&	8.87e-07&	-7.62e-01&	2.40e-06\\
2&	13852&	-0.392&	1.03&	8.68e-02&	4.96e-06&	-4.78e-01&	2.39e-06\\
3&	17353&	-0.275&	1.02&	1.17e-02&	-1.42e-06&	-2.87e-01&	2.38e-06\\
4&	22732&	-0.206&	   1.00&	-2.97e-02&	-1.47e-06&	-1.77e-01&	2.38e-06\\
5&	33036&	-0.156&	0.99&	-5.19e-02&	-8.73e-07&	-1.04e-01&	2.38e-06\\
6&	50796&	-0.127&	0.98&	-6.95e-02&	-1.01e-06&	-5.80e-02&	2.38e-06\\
7&	86276&	1.67e-05&	-&	3.46e-02&	-8.72e-07&	-3.45e-02&	2.38e-06\\
\bottomrule
\end{tabular}
\caption{Adaptive Contribution Refinement strategy defined in \S \ref{sec:refine_strategies_apost_ideas} applied to the nonlinear PBE for methanol. }
\label{tab:adapt_comp_meth_nonlin}
\end{table}

\subsubsection{Results for Classical Refinement}

The results of the Classical Refinement strategy defined in \S \ref{sec:refine_strategies_apost_ideas} for the solution of the linearized and nonlinear PBE for the Born ion are shown in Tables~\ref{tab:dwr_born_lin} and \ref{tab:dwr_born_nonlin}, while the results for methanol are shown in Tables~\ref{tab:dwr_meth_lin} and \ref{tab:dwr_meth_nonlin}. Comparing these results to Tables~\ref{tab:born_lin_1}, \ref{tab:born_nonlin_1}, \ref{tab:meth_lin_1} and \ref{tab:meth_nonlin_1}, we observe that 
Classical refinement also performs well compared to uniform refinement. However, its performance is slightly worse than Adaptive Contribution Refinement as illustrated by Tables~\ref{tab:adapt_comp_born_lin} and \ref{tab:dwr_born_lin}. In fact, the error at level 7 in table \ref{tab:dwr_born_lin} shows almost a doubling of error at level 6. This is explained by observing the behavior of the terms $E^r$ and $E^\Gamma$, which are the two dominant sources of error, at these levels. Although both terms decrease in magnitude, there is less cancellation of error, leading to an overall increase. A similar increase in the error is observed at level 7 of Table~\ref{tab:dwr_born_nonlin}.

\begin{table}[!ht]
\centering
\begin{tabular}{c|c|c|c|c|c|c|c}
\toprule
It. &$N$ & Est. Err. & $\gamma_{\rm eff}$ &  $E^r$ & $E^m$ & $E^\Gamma$   & $E^{\partial \Omega}$ \\
\midrule
0&	6718&	-1.14&	1.05&	2.05e-01&	5.26e-09&	-1.34e+00&	4.86e-04\\
1&	9481&	-1.07&	1.03&	1.50e-01&	5.62e-08&	-1.22e+00&	4.79e-04\\
2&	14299&	-0.512&	1.02&	1.24e-01&	-4.16e-08&	-6.36e-01&	4.78e-04\\
3&	21638&	-0.23&	1.02&	1.06e-01&	2.82e-07&	-3.37e-01&	4.78e-04\\
4&	33959&	-0.0868&	1.02&	9.58e-02&	1.83e-07&	-1.83e-01&	4.78e-04\\
5&	55762&	-0.0116&	   1.00&	8.68e-02&	9.50e-08&	-9.89e-02&	4.78e-04\\
6&	95162&	0.0275&	1.03&	8.10e-02&	1.28e-08&	-5.40e-02&	4.78e-04\\
7&	165202&	0.0487&	1.02&	7.75e-02&	9.00e-08&	-2.94e-02&	4.78e-04\\
\bottomrule
\end{tabular}
\caption{Classical Refinement strategy defined in \S \ref{sec:refine_strategies_apost_ideas} applied to the linearized PBE for the Born ion. }
\label{tab:dwr_born_lin}
\end{table}

\begin{table}[!ht]
\centering
\begin{tabular}{c|c|c|c|c|c|c|c}
\toprule
It. &$N$ & Est. Err. & $\gamma_{\rm eff}$ &  $E^r$ & $E^m$ & $E^\Gamma$   & $E^{\partial \Omega}$ \\
\midrule
0&	6718&	-1.16&	1.05&	1.88e-01&	5.13e-09&	-1.34e+00&	4.29e-04\\
1&	9481&	-1.08&	1.02&	1.38e-01&	1.06e-07&	-1.22e+00&	4.19e-04\\
2&	14299&	-0.521&	1.02&	1.14e-01&	-3.29e-08&	-6.36e-01&	4.17e-04\\
3&	21638&	-0.239&	1.02&	9.70e-02&	1.84e-07&	-3.37e-01&	4.16e-04\\
4&	33959&	-0.0955&	1.01&	8.71e-02&	1.99e-07&	-1.83e-01&	4.16e-04\\
5&	55762&	-0.02&	0.98&	7.85e-02&	1.07e-07&	-9.89e-02&	4.16e-04\\
6&	95162&	0.0192&	1.06&	7.28e-02&	2.90e-07&	-5.40e-02&	4.16e-04\\
7&	165192&	0.0404&	1.04&	6.93e-02&	4.39e-08&	-2.94e-02&	4.16e-04\\
\bottomrule
\end{tabular}
\caption{Classical Refinement strategy defined in \S \ref{sec:refine_strategies_apost_ideas} applied to the nonlinear PBE for the Born ion. }
\label{tab:dwr_born_nonlin}
\end{table}

\begin{table}[!ht]
\centering
\begin{tabular}{c|c|c|c|c|c|c|c}
\toprule
It. &$N$ & Est. Err. & $\gamma_{\rm eff}$ &  $E^r$ & $E^m$ & $E^\Gamma$   & $E^{\partial \Omega}$ \\
\midrule
0&	11769&	-0.924&	1.03&	4.27e-01&	4.67e-06&	-1.35e+00&	2.43e-06\\
1&	12637&	-0.613&	1.03&	1.25e-01&	7.97e-07&	-7.38e-01&	2.41e-06\\
2&	14798&	-0.352&	1.02&	5.26e-02&	4.42e-07&	-4.04e-01&	2.40e-06\\
3&	20782&	-0.267&	1.01&	-5.43e-02&	-1.61e-06&	-2.13e-01&	2.40e-06\\
4&	34556&	-0.131&	0.99&	-3.32e-02&	6.86e-08&	-9.74e-02&	2.40e-06\\
5&	63134&	-0.0388&	0.94&	1.25e-02&	-2.33e-07&	-5.12e-02&	2.40e-06\\
6&	118378&	-0.0098&	-&	1.77e-02&	-1.35e-07&	-2.75e-02&	2.40e-06\\
\bottomrule
\end{tabular}

\caption{Classical Refinement strategy defined in \S \ref{sec:refine_strategies_apost_ideas} applied to the linearized PBE for Methanol. }
\label{tab:dwr_meth_lin}
\end{table}

\begin{table}[!ht]
\centering
\begin{tabular}{c|c|c|c|c|c|c|c}
\toprule
It. &$N$ & Est. Err. & $\gamma_{\rm eff}$ &  $E^r$ & $E^m$ & $E^\Gamma$   & $E^{\partial \Omega}$ \\
\midrule
0&	11769&	-0.924&	1.03&	4.27e-01&	3.31e-06&	-1.35e+00&	2.41e-06\\
1&	12637&	-0.613&	1.03&	1.24e-01&	2.05e-06&	-7.38e-01&	2.40e-06\\
2&	14798&	-0.352&	1.02&	5.22e-02&	-1.86e-07&	-4.04e-01&	2.38e-06\\
3&	20782&	-0.267&	1.01&	-5.47e-02&	-3.94e-07&	-2.13e-01&	2.38e-06\\
4&	34538&	-0.131&	0.98&	-3.35e-02&	-3.19e-07&	-9.75e-02&	2.38e-06\\
5&	63070&	-0.0391&	0.94&	1.22e-02&	1.45e-07&	-5.13e-02&	2.38e-06\\
6&	118208&	-0.01&	-&	1.75e-02&	5.36e-08&	-2.75e-02&	2.38e-06\\
\bottomrule
\end{tabular}
\caption{Classical Refinement strategy defined in \S \ref{sec:refine_strategies_apost_ideas} applied to the nonlinear PBE for methanol. }
\label{tab:dwr_meth_nonlin}
\end{table}

\subsubsection{Experiment illustrating difference between linear and nonlinear PBE results}
\label{sec:born_spiked}

In this section we perform an experiment to illustrate the difference in the results of the  linear and nonlinear PBE solutions. To this end, we again choose the Born ion but now the charge on the ion, $Q_1$, is taken to be ten times its value in earlier experiments and also set $\okappa^2 = 9.18168$ which is also ten times larger than in earlier experiments. We call this setup  the highly charged Born ion. The difference in the computed QoI between the linear and nonlinear PBE for a mesh of 6718 vertices was approximately 59 units. The results for the different adaptive strategies also indicate different behavior between the linearized and nonlinear PBE.



The results for the linear and nonlinear PBE using
Uniform Contribution refinement, Adaptive Contribution refinement and Classical Refinement are shown in Tables 
\ref{tab:UCR_spiked_born_lin}, \ref{tab:UCR_spiked_born_nonlin}, \ref{tab:ACR_spiked_born_lin}, \ref{tab:ACR_spiked_born_nonlin}, \ref{tab:dwr_spiked_born_lin}, \ref{tab:dwr_spiked_born_nonlin}. The results indicate that the Adaptive Contribution Refinement performs better than Classical Refinement for the linearized PBE, while they both perform equally well for the nonlinear PBE. Uniform Contribution Refinement outperforms both Classical Refinement and Adaptive Contribution Refinement.




\begin{table}[!ht]
\centering
\begin{tabular}{c|c|c|c|c|c|c|c}
\toprule
It. &$N$ & Est. Err. & $\gamma_{\rm eff}$ &  $E^r$ & $E^m$ & $E^\Gamma$   & $E^{\partial \Omega}$ \\
\midrule
0&	6718&	-128&	1.04&	6.81e+00&	1.92e-07&	-1.35e+02&	3.76e-14\\
1&	15541&	-28.2&	1.04&	8.51e+00&	7.96e-05&	-3.68e+01&	4.10e-15\\
2&	41760&	-5.12&	1.01&	3.95e+00&	1.94e-05&	-9.07e+00&	3.14e-15\\
3&	141855&	-0.0159&	-&	2.24e+00&	1.66e-05&	-2.25e+00&	3.12e-15\\
\bottomrule
\end{tabular}
\caption{Uniform Contribution Refinement strategy applied to linearized PBE for the setup described in \S \ref{sec:born_spiked}.}
\label{tab:UCR_spiked_born_lin}
\end{table}

\begin{table}[!ht]
\centering
\begin{tabular}{c|c|c|c|c|c|c|c}
\toprule
It. &$N$ & Est. Err. & $\gamma_{\rm eff}$ &  $E^r$ & $E^m$ & $E^\Gamma$   & $E^{\partial \Omega}$ \\
\midrule
0&	6718&	-131&	1.04&	3.72e+00&	3.27e-07&	-1.35e+02&	2.25e-15\\
1&	15541&	-29.4&	1.04&	7.36e+00&	7.20e-05&	-3.68e+01&	2.46e-16\\
2&	41760&	-6.24&	1.01&	2.83e+00&	2.36e-05&	-9.07e+00&	1.88e-16\\
3&	141855&	-1.18&	-&	1.07e+00&	-5.90e-07&	-2.25e+00&	1.83e-16\\
\bottomrule
\end{tabular}
\caption{Uniform Contribution Refinement strategy applied to nonlinear PBE for the setup described in \S \ref{sec:born_spiked}.}
\label{tab:UCR_spiked_born_nonlin}
\end{table}

\begin{table}[!ht]
\centering
\begin{tabular}{c|c|c|c|c|c|c|c}
\toprule
It. &$N$ & Est. Err. & $\gamma_{\rm eff}$ &  $E^r$ & $E^m$ & $E^\Gamma$   & $E^{\partial \Omega}$ \\
\midrule
0&	6718&	-128&	1.04&	6.81e+00&	1.92e-07&	-1.35e+02&	3.76e-14\\
1&	9458&	-116&	1.02&	6.51e+00&	9.61e-06&	-1.22e+02&	8.52e-15\\
2&	14164&	-59.5&	1.02&	5.72e+00&	7.12e-05&	-6.53e+01&	4.42e-15\\
3&	21205&	-29.9&	1.01&	4.49e+00&	-2.03e-05&	-3.44e+01&	3.30e-15\\
4&	33141&	-15.3&	1.01&	3.55e+00&	2.44e-05&	-1.89e+01&	3.29e-15\\
5&	53605&	-7.32&	0.99&	2.85e+00&	-4.94e-06&	-1.02e+01&	3.28e-15\\
6&	91792&	-3.3&	0.95&	2.30e+00&	2.61e-05&	-5.59e+00&	3.15e-15\\
7&	160005&	-1.12&	-&	1.92e+00&	1.41e-05&	-3.04e+00&	3.15e-15\\
\bottomrule
\end{tabular}
\caption{Adaptive Contribution Refinement strategy applied to linearized PBE for the setup described in \S \ref{sec:born_spiked}.}
\label{tab:ACR_spiked_born_lin}
\end{table}

\begin{table}[!ht]
\centering
\begin{tabular}{c|c|c|c|c|c|c|c}
\toprule
It. &$N$ & Est. Err. & $\gamma_{\rm eff}$ &  $E^r$ & $E^m$ & $E^\Gamma$   & $E^{\partial \Omega}$ \\
\midrule
0&	6718&	-131&	1.04&	3.72e+00&	3.27e-07&	-1.35e+02&	2.25e-15\\
1&	9458&	-118&	1.02&	4.23e+00&	6.37e-06&	-1.22e+02&	5.13e-16\\
2&	14164&	-60.8&	1.02&	4.51e+00&	1.48e-05&	-6.53e+01&	2.68e-16\\
3&	21206&	-30.9&	1.01&	3.59e+00&	2.14e-06&	-3.45e+01&	1.98e-16\\
4&	33216&	-16.2&	   1.00&	2.63e+00&	2.77e-05&	-1.88e+01&	1.97e-16\\
5&	53748&	-8.24&	0.98&	1.90e+00&	-3.05e-06&	-1.01e+01&	1.98e-16\\
6&	92219&	-4.27&	0.95&	1.30e+00&	2.44e-06&	-5.56e+00&	1.88e-16\\
7&	160740&	-2.13&	0.90&	8.86e-01&	-2.27e-06&	-3.02e+00&	1.87e-16\\
\bottomrule
\end{tabular}
\caption{Adaptive Contribution Refinement strategy applied to nonlinear PBE for the setup described in \S \ref{sec:born_spiked}.}
\label{tab:ACR_spiked_born_nonlin}
\end{table}

\begin{table}[!ht]
\centering
\begin{tabular}{c|c|c|c|c|c|c|c}
\toprule
It. &$N$ & Est. Err. & $\gamma_{\rm eff}$ &  $E^r$ & $E^m$ & $E^\Gamma$   & $E^{\partial \Omega}$ \\
\midrule
0&	6718&	-128&	1.04&	6.81e+00&	1.92e-07&	-1.35e+02&	3.76e-14\\
1&	9481&	-116&	1.02&	6.51e+00&	9.62e-06&	-1.22e+02&	8.52e-15\\
2&	14299&	-57.9&	1.02&	5.69e+00&	7.71e-05&	-6.36e+01&	4.10e-15\\
3&	21638&	-29.2&	1.01&	4.46e+00&	1.38e-05&	-3.37e+01&	3.60e-15\\
4&	33952&	-14.8&	1.01&	3.49e+00&	1.23e-05&	-1.83e+01&	3.58e-15\\
5&	55747&	-7.09&	0.98&	2.80e+00&	2.70e-05&	-9.89e+00&	3.03e-15\\
6&	95115&	-3.16&	0.95&	2.25e+00&	1.60e-05&	-5.40e+00&	2.91e-15\\
\bottomrule
\end{tabular}
\caption{Classical Refinement strategy applied to linearized PBE for the setup described in \S \ref{sec:born_spiked}.}
\label{tab:dwr_spiked_born_lin}
\end{table}

\begin{table}[!ht]
\centering
\begin{tabular}{c|c|c|c|c|c|c|c}
\toprule
It. &$N$ & Est. Err. & $\gamma_{\rm eff}$ &  $E^r$ & $E^m$ & $E^\Gamma$   & $E^{\partial \Omega}$ \\
\midrule
0&	6718&	-131&	1.04&	3.72e+00&	3.27e-07&	-1.35e+02&	2.25e-15\\
1&	9477&	-118&	1.02&	4.23e+00&	8.49e-06&	-1.22e+02&	5.12e-16\\
2&	14286&	-59.8&	1.02&	4.49e+00&	1.86e-05&	-6.43e+01&	2.47e-16\\
3&	21365&	-30.4&	1.01&	3.61e+00&	2.82e-05&	-3.40e+01&	2.21e-16\\
4&	33470&	 -16&	   1.00&	2.60e+00&	1.84e-05&	-1.86e+01&	2.18e-16\\
5&	54450&	-8.16&	0.98&	1.88e+00&	6.44e-07&	-1.00e+01&	1.85e-16\\
6&	93369&	-4.22&	0.95&	1.28e+00&	-5.89e-07&	-5.50e+00&	1.77e-16\\
\bottomrule
\end{tabular}
\caption{Classical Refinement strategy applied to nonlinear PBE for the setup described in \S \ref{sec:born_spiked}.}
\label{tab:dwr_spiked_born_nonlin}
\end{table}


\section{Conclusions}
\label{sec:conclusions}
Computing a QoI from the numerical solution of the PBE often has significant error that needs to be quantified. In this article, we develop adjoint based error estimates for this purpose. The adjoint operators are defined by accounting for the coupled nature of the three term split PBE as well as the issues arising due to the regularity of the normal derivative. The resulting error estimates are shown to be accurate, with effectivity ratios close to one. The error is partitioned in such a way that specific sources of error are identified and addressed. Moreover, novel refinement schemes, called Uniform Contribution Refinement and Adaptive Contribution Refinement in this article, utilize the information about the sources of error to arrive at accurate computed values of the QoI.

The effects of interface geometry on the error is an interesting area of future research. The current article is based on the the standard assumption in the  PBE literature that the tessellated geometric representation of the interface is the true interface, e.g. as in references~\cite{ABC+2011,BHW00,CHX07}. The effect of the geometry, which could be considered a ``model form error'', is an interesting topic to explore and the author intends to pursue it in future.

\clearpage

\bibliographystyle{plain}
\bibliography{pbe_apost}

\begin{thebibliography}{10}

\bibitem{AO2000}
M.~Ainsworth and T.~Oden.
\newblock {\em \emph{A posteriori} error estimation in finite element
  analysis}.
\newblock John Wiley-Teubner, 2000.

\bibitem{ABC+2011}
Burak Aksoylu, Stephen~D. Bond, Eric~C. Cyr, and Michael Holst.
\newblock Goal-oriented adaptivity and multilevel preconditioning for the
  {P}oisson-{B}oltzmann equation.
\newblock {\em Journal of Scientific Computing}, 52(1):202--225, Oct 2011.

\bibitem{AlnaesBlechta2015a}
Martin~S. Aln{\ae}s, Jan Blechta, Johan Hake, August Johansson, Benjamin
  Kehlet, Anders Logg, Chris Richardson, Johannes Ring, Marie~E. Rognes, and
  Garth~N. Wells.
\newblock The {FEniCS} project version 1.5.
\newblock {\em Archive of Numerical Software}, 3(100), 2015.

\bibitem{BBC06}
N.~A. Baker, D.~Bashford, and D.~A. Case.
\newblock Implicit solvent electrostatics in biomolecular simulation.
\newblock In Benedict Leimkuhler, Christophe Chipot, Ron Elber, Aatto
  Laaksonen, Alan Mark, Tamar Schlick, Christof Schutte, and Robert Skeel,
  editors, {\em New Algorithms for Macromolecular Simulation}, volume~49 of
  {\em Lecture Notes in Computational Science and Engineering}, pages 263--295.
  Springer-Verlag, 2006.

\bibitem{BHW00}
N.~A. Baker, M.~J. Holst, and F.~Wang.
\newblock Adaptive multilevel finite element solution of the
  {P}oisson-{B}oltzmann equation {II}: Refinement at solvent accessible
  surfaces in biomolecular systems.
\newblock {\em J. Comput.\ Chem.}, 21:1343--1352, 2000.

\bibitem{BSJ+2001}
N.~A. Baker, D.~Sept, S.~Joseph, M.~J. Holst, and J.~A. McCammon.
\newblock Electrostatics of nanosystems: Application to microtubules and the
  ribosome.
\newblock {\em Proceedings of the National Academy of Sciences},
  98(18):10037--10041, Aug 2001.

\bibitem{BangerthRannacher}
W.~Bangerth and R.~Rannacher.
\newblock {\em Adaptive Finite Element Methods for Differential Equations}.
\newblock Birkhauser Verlag, 2003.

\bibitem{barth04}
T.~J. Barth.
\newblock {\em \emph{A posteriori} Error Estimation and Mesh Adaptivity for
  Finite Volume and Finite Element Methods}, volume~41 of {\em Lecture Notes in
  Computational Science and Engineering}.
\newblock Springer, New York, 2004.

\bibitem{Bashford97}
D.~Bashford.
\newblock An object-oriented programming suite for electrostatic effects in
  biological molecules. {A}n experience report on the {MEAD} project.
\newblock In {\em Scientific Computing in Object-Oriented Parallel
  Environments}, volume 1343 of {\em Lecture Notes in Computer Science}, pages
  233--240, 1997.

\bibitem{beckerrannacher}
R.~Becker and R.~Rannacher.
\newblock An optimal control approach to \emph{a posteriori} error estimation
  in finite element methods.
\newblock {\em Acta Numerica}, pages 1--102, 2001.

\bibitem{BWSH95}
Ranganathan Bharadwaj, Andreas Windemuth, S.~Sridharan, Barry Honig, and
  Anthony Nicholls.
\newblock The fast multipole boundary element method for molecular
  electrostatics: An optimal approach for large systems.
\newblock {\em J. Comput.\ Chem.}, 16(7):898--913, 1995.

\bibitem{BCCO09}
Stephen~D. Bond, Jehanzeb~Hameed Chaudhry, Eric~C. Cyr, and Luke~N. Olson.
\newblock A first-order system least-squares finite element method for the
  {P}oisson-{B}oltzmann equation.
\newblock {\em J.~Comput.~Chem.}, 31(8):1625--1635, 2010.

\bibitem{YCLP2004}
Yang Cao and Linda Petzold.
\newblock A posteriori error estimation and global error control for ordinary
  differential equations by the adjoint method.
\newblock {\em SIAM Journal on Scientific Computing}, 26(2):359--374, 2004.

\bibitem{CET-09}
V.~Carey, D.~Estep, and S.J. Tavener.
\newblock A posteriori analysis and adaptive error control for multiscale
  operator decomposition solution of elliptic systems {I}: One way coupled
  systems.
\newblock {\em SIAM Journal on Numerical Analysis}, 47(1):740--761, 2009.

\bibitem{Ch1913}
D.~L. Chapman.
\newblock A contribution to the theory of electrocapillarity.
\newblock {\em Philos. Mag.}, 25:475--481, 1913.

\bibitem{CEG+2015}
Jehanzeb~H Chaudhry, Donald Estep, Victor Ginting, John~N Shadid, and Simon
  Tavener.
\newblock A posteriori error analysis of imex multi-step time integration
  methods for advection--diffusion--reaction equations.
\newblock {\em Computer Methods in Applied Mechanics and Engineering},
  285:730--751, 2015.

\bibitem{Chaudhry2011}
Jehanzeb~Hameed Chaudhry, Stephen~D. Bond, and Luke~N. Olson.
\newblock Finite element approximation to a finite-size modified
  {P}oisson-{B}oltzmann equation.
\newblock {\em Journal of Scientific Computing}, 47(3):347--364, 2011.

\bibitem{CBO12}
Jehanzeb~Hameed Chaudhry, Stephen~D. Bond, and Luke~N. Olson.
\newblock A weighted adaptive least-squares finite element method for the
  {P}oisson-{B}oltzmann equation.
\newblock {\em Applied Mathematics and Computation}, 218(9):4892 -- 4902, 2012.

\bibitem{CET+2016}
Jehanzeb~Hameed Chaudhry, Don Estep, Simon Tavener, Varis Carey, and Jeff
  Sandelin.
\newblock A posteriori error analysis of two-stage computation methods with
  application to efficient discretization and the {P}arareal algorithm.
\newblock {\em SIAM Journal on Numerical Analysis}, 54(5):2974--3002, 2016.

\bibitem{CHX07}
Long Chen, Michael~J. Holst, and Jinchao Xu.
\newblock The finite element approximation of the nonlinear
  {P}oisson-{B}oltzmann equation.
\newblock {\em SIAM J. Numer. Anal.}, 45(6):2298--2320, 2007.

\bibitem{WYQ05}
Wenbin Chen, Yifan Shen, and Qing Xia.
\newblock A mortar finite element approximation for the linear
  {P}oisson-{B}oltzmann equation.
\newblock {\em Appl. Math. Comput.}, 164(1):11--23, 2005.

\bibitem{jdt13}
J.~B. Collins, D.~Estep, and S.~Tavener.
\newblock A posteriori error analysis for finite element methods with
  projection operators as applied to explicit time integration techniques.
\newblock {\em BIT Numerical Mathematics}, 55(4):1017--1042, 2015.

\bibitem{CBH+2014}
Jeffrey~M Connors, Jeffrey~W Banks, Jeffrey~A Hittinger, and Carol~S Woodward.
\newblock Quantification of errors for operator-split advection--diffusion
  calculations.
\newblock {\em Computer Methods in Applied Mechanics and Engineering},
  272:181--197, 2014.

\bibitem{CoFr97}
Christina~M. Cortis and Richard~A. Friesner.
\newblock Numerical solution of the {P}oisson-{B}oltzmann equation using
  tetrahedral finite-element meshes.
\newblock {\em J. Comput.\ Chem.}, 18:1591--1608, 1997.

\bibitem{DaMc89}
M.~E. Davis and J.~A. McCammon.
\newblock Solving the finite difference linearized {P}oisson-{B}oltzmann
  equation: A comparison of relaxation and conjugate gradient methods.
\newblock {\em J. Comput.\ Chem.}, 10:386--391, 1989.

\bibitem{DaMc90}
Malcolm~E. Davis and J.~Andrew Mc{C}ammon.
\newblock Electrostatics in biomolecular structure and dynamics.
\newblock {\em Chem. Rev.}, 90(3):509--521, 1990.

\bibitem{Do96}
W.~D\"{o}rfler.
\newblock A convergent adaptive algorithm for {P}oisson's equation.
\newblock {\em SIAM Journal on Numerical Analysis}, 33(3):1106--1124, 1996.

\bibitem{eehj_actanum_95}
K.~Eriksson, D.~Estep, P.~Hansbo, and C.~Johnson.
\newblock Introduction to adaptive methods for differential equations.
\newblock In {\em Acta Numerica, 1995}, Acta Numerica, pages 105--158.
  Cambridge Univ. Press, Cambridge, 1995.

\bibitem{eehj_book_96}
K.~Eriksson, D.~Estep, P.~Hansbo, and C.~Johnson.
\newblock {\em Computational Differential Equations}.
\newblock Cambridge University Press, Cambridge, 1996.

\bibitem{estep_sinum_95}
D.~Estep.
\newblock \emph{A posteriori} error bounds and global error control for
  approximation of ordinary differential equations.
\newblock {\em SIAM J. Numer. Anal.}, 32(1):1--48, 1995.

\bibitem{Fish:2009}
D.~Estep.
\newblock Error estimates for multiscale operator decomposition for
  multiphysics models.
\newblock In J.~Fish, editor, {\em Multiscale methods: bridging the scales in
  science and engineering}. Oxford University Press, USA, 2009.

\bibitem{Estep:Larson:00}
D.~J. Estep, M.~G. Larson, R.~D. Williams, and American~Mathematical Society.
\newblock {\em Estimating the error of numerical solutions of systems of
  reaction-diffusion equations}.
\newblock American Mathematical Society, 2000.

\bibitem{EHL05}
Donald Estep, Michael Holst, and Mats Larson.
\newblock Generalized {G}reen's functions and the effective domain of
  influence.
\newblock {\em SIAM Journal on Scientific Computing}, 26(4):1314--1339, 2005.

\bibitem{ETW2010}
Donald Estep, Simon Tavener, and Tim Wildey.
\newblock A posteriori error estimation and adaptive mesh refinement for a
  multiscale operator decomposition approach to fluid--solid heat transfer.
\newblock {\em Journal of Computational Physics}, 229(11):4143--4158, 2010.

\bibitem{FBM02}
F.~Fogolari, A.~Brigo, and H.~Molinari.
\newblock The {P}oisson-{B}oltzmann equation for biomolecular electrostatics: A
  tool for structural biology.
\newblock {\em J. Mol. Recognit.}, 15(6):377--392, 2002.

\bibitem{GYW2007}
W.~Geng, S.~Yu, and G.~Wei.
\newblock Treatment of charge singularities in implicit solvent models.
\newblock {\em J.\ Chem. Phys.}, 127(11):114106, 2007.

\bibitem{Giles:Suli:02}
M.~B. Giles and E.~S{\"u}li.
\newblock Adjoint methods for {PDE}s: a posteriori error analysis and
  postprocessing by duality.
\newblock {\em Acta Numerica}, 11(1):145--236, 2002.

\bibitem{GSH87}
M.~K. Gilson, K.~A. Sharp, and B.~H. Honig.
\newblock Calculating the electrostatic potential of molecules in solution:
  Method and error assessment.
\newblock {\em J. Comput.\ Chem.}, 9:327--335, 1987.

\bibitem{Go1910}
G.~Gouy.
\newblock Sur la constitution de la charge \'electrique a la surface d'un
  \'electrolyte.
\newblock {\em J. Phys. Theor. Appl.}, 9:457--468, 1910.

\bibitem{Hols01}
M.~Holst.
\newblock Adaptive numerical treatment of elliptic systems on manifolds.
\newblock {\em Adv. Comput. Math.}, 15(1--4):139--191, 2001.

\bibitem{HMY+2012}
M.~Holst, J.A. McCammon, Z.~Yu, Y.C. Zhou, and Y.~Zhu.
\newblock Adaptive finite element modeling techniques for the
  {P}oisson-{B}oltzmann equation.
\newblock {\em Communications in Computational Physics}, 11(01):179--214, Jan
  2012.

\bibitem{HBW00}
M.~J. Holst, N.~A. Baker, and F.~Wang.
\newblock Adaptive multilevel finite element solution of the
  {P}oisson-{B}oltzmann equation {I}: Algorithms and examples.
\newblock {\em J. Comput.\ Chem.}, 21:1319--1342, 2000.

\bibitem{HoSa95}
M.~J. Holst and F.~Saied.
\newblock Numerical solution of the nonlinear {P}oisson-{B}oltzmann equation:
  Developing more robust and efficient methods.
\newblock {\em J. Comput.\ Chem.}, 16:337--364, 1995.

\bibitem{Kato13}
Tosio Kato.
\newblock {\em Perturbation Theory for Linear Operators}.
\newblock Springer, 2013.

\bibitem{Kirk34}
John~G. Kirkwood.
\newblock Theory of solutions of molecules containing widely separated charges
  with special application to zwitterions.
\newblock {\em J. Chem.\ Phys.}, 2(7):351--361, 1934.

\bibitem{Ko2006}
Patrice Koehl.
\newblock Electrostatics calculations: Latest methodological advances.
\newblock {\em Curr. Opin. Struc. Biol.}, 16(2):142--151, 2006.

\bibitem{LoggMardalEtAl2012a}
Anders Logg, Kent-Andre Mardal, Garth~N. Wells, et~al.
\newblock {\em Automated Solution of Differential Equations by the Finite
  Element Method}.
\newblock Springer, 2012.

\bibitem{LoggWells2010a}
Anders Logg and Garth~N. Wells.
\newblock {DOLFIN}: Automated finite element computing.
\newblock {\em ACM Transactions on Mathematical Software}, 37(2), 2010.

\bibitem{LoggWellsEtAl2012a}
Anders Logg, Garth~N. Wells, and Johan Hake.
\newblock {\em {DOLFIN}: a C++/Python Finite Element Library}, chapter~10.
\newblock Springer, 2012.

\bibitem{LZMJ05}
B.~Lu, D.~Zhang, and J.~A. McCammon.
\newblock Computation of electrostatic forces between solvated molecules
  determined by the {P}oisson-{B}oltzmann equation using a boundary element
  method.
\newblock {\em J.\ Chem. Phys.}, 122(21):214102, 2005.

\bibitem{LDM92}
B.~A. Luty, M.~E. Davis, and J.~A. McCammon.
\newblock Solving the finite-difference non-linear {P}oisson-{B}oltzmann
  equation.
\newblock {\em J. Comput.\ Chem.}, 13(9):1114--1118, 1992.

\bibitem{Marchuk1995}
Guri~I. Marchuk.
\newblock {\em Adjoint Equations and Analysis of Complex Systems}.
\newblock Springer Nature, 1995.

\bibitem{MAS1996}
Guri~I Marchuk, Valeri~I Agoshkov, and Victor~P Shutyaev.
\newblock {\em Adjoint equations and perturbation algorithms in nonlinear
  problems}.
\newblock CRC Press, 1996.

\bibitem{McQuarrie76}
Donald~A. McQuarrie.
\newblock {\em Statistical Mechanics}.
\newblock Harpercollins College Div, 1976.

\bibitem{NA2007}
Marian Nemec and Michael Aftosmis.
\newblock Adjoint error estimation and adaptive refinement for
  embedded-boundary {C}artesian meshes.
\newblock In {\em 18th AIAA Computational Fluid Dynamics Conference}, page
  4187, 2007.

\bibitem{Ortt77}
William~H. Orttung.
\newblock Direct solution of the {P}oisson equation for biomolecules of
  arbitrary shape, polarizability density, and charge distribution.
\newblock {\em Ann. N.Y. Acad. Sci.}, 303:22--37, 1977.

\bibitem{VRAS2015}
Vishwas Rao and Adrian Sandu.
\newblock A posteriori error estimates for the solution of variational inverse
  problems.
\newblock {\em SIAM/ASA Journal on Uncertainty Quantification}, 3(1):737--761,
  2015.

\bibitem{Ra90}
A.~A. Rashin.
\newblock Hydration phenomena, classical electrostatics, and the boundary
  element method.
\newblock {\em J.\ Chem. Phys.}, 94(5):1725--1733, 1990.

\bibitem{SMN02}
A.~I. Shestakov, J.~L. Milovich, and A.~Noy.
\newblock Solution of the nonlinear {P}oisson-{B}oltzmann equation using
  pseudo-transient continuation and the finite element method.
\newblock {\em J. Colloid Interf. Sci.}, 247(1):62--79, 2002.

\bibitem{Tanford61}
Charles Tanford.
\newblock {\em Physical chemistry of macromolecules}.
\newblock New York, Wiley, 1961.

\bibitem{ViMa2005}
Christina~L. Vizcarra and Stephen~L. Mayo.
\newblock Electrostatics in computational protein design.
\newblock {\em Curr. Opin. Chem. Biol.}, 9(6):622--626, 2005.

\bibitem{VGS92}
Y.~N. Vorobjev, J.~A. Grant, and H.~A. Scheraga.
\newblock A combined iterative and boundary-element approach for solution of
  the nonlinear {P}oisson-{B}oltzmann equation.
\newblock {\em J. Amer. Chem. Soc.}, 114(9):3189--3196, April 1992.

\bibitem{VoSc97}
Yury~N. Vorobjev and Harold~A. Scheraga.
\newblock A fast adaptive multigrid boundary element method for macromolecular
  electrostatic computations in a solvent.
\newblock {\em J. Comput.\ Chem.}, 18(4):569--583, 1997.

\bibitem{YoLe90}
B.~J. Yoon and A.~M. Lenhoff.
\newblock A boundary element method for molecular electrostatics with
  electrolyte effects.
\newblock {\em J.\ Comput. Chem.}, 11(9):1080--1086, 1990.

\bibitem{Yosida08}
Kôsaku Yosida.
\newblock {\em Functional Analysis}.
\newblock Springer, 2008.

\bibitem{YHCM08}
Z.~Yu, M.~J. Holst, Y.~Cheng, and J.~A. Mc{C}ammon.
\newblock Feature-preserving adaptive mesh generation for molecular shape
  modeling and simulation.
\newblock {\em J. Mol. Graph. Model.}, 26(8):1370--1380, 2008.

\bibitem{Zho93}
H.-X. Zhou.
\newblock Boundary element solution of macromolecular electrostatics:
  Interaction energy between two proteins.
\newblock {\em Biophys.\ J.}, 65:955--963, 1993.

\bibitem{ZPVK96}
Zhongxiang Zhou, Philip Payne, Max Vasquez, Nat Kuhn, and Michael Levitt.
\newblock Finite-difference solution of the {P}oisson-{B}oltzmann equation:
  Complete elimination of self-energy.
\newblock {\em J. Comput.\ Chem.}, 17(11):1344--1351, 1996.

\end{thebibliography}

\end{document}